\documentclass[%
  a4paper,
  onecolumn,
  colorlinks,
]{preprint}

\usepackage{geometry}
\usepackage[normalem]{ulem}
\usepackage{graphicx} 
\usepackage{amsmath}
\usepackage{amsthm}
\usepackage{amssymb}
\usepackage{xcolor}

\usepackage{algorithm}
\usepackage{algpseudocode}

\usepackage{subcaption}
\usepackage{enumitem}
\usepackage{pgfplots}
\usetikzlibrary{positioning,arrows.meta}

\algrenewcommand\algorithmicrequire{\textbf{Input:}}
\algrenewcommand\algorithmicensure{\textbf{Output:}}

\newcommand{\R}{\mathbb{R}}
\newcommand{\C}{\mathbb{C}}
\renewcommand{\H}{\mathbb{H}}
\renewcommand{\L}{\mathbb{L}}

\newcommand{\D}{\mathbb{D}}
\newcommand{\N}{\mathbb{N}}
\newcommand{\A}{\mathbb{A}}
\newcommand{\B}{\mathbb{B}}
\newcommand{\M}{\mathbb{M}}

\DeclareMathOperator*{\argmax}{arg\,max}

\DeclareMathOperator*{\diag}{diag}
\DeclareMathOperator*{\vectorize}{vec}

\newcommand{\bc}{\textbf{c}}
\newcommand{\bd}{\textbf{d}}

\newcommand{\bff}{\textbf{f}}

\newcommand{\br}{\textbf{r}}

\newcommand{\bn}{\textbf{n}}
\newcommand{\bu}{\textbf{u}}

\newcommand{\bx}{\textbf{x}}
\newcommand{\by}{\textbf{y}}
\newcommand{\bz}{\textbf{z}}
\newcommand{\Mu}{\mathrm{M}}

\newcommand{\bD}{\textbf{D}}

\newcommand{\bF}{\textbf{F}}

\newcommand{\bH}{\textbf{H}}
\newcommand{\bI}{\textbf{I}}

\newcommand{\bM}{\textbf{M}}

\newcommand{\bS}{\textbf{S}}

\newcommand{\bX}{\textbf{X}}
\newcommand{\bY}{\textbf{Y}}
\newcommand{\bZ}{\textbf{Z}}

\newcommand{\blambda}{\boldsymbol{\lambda}}
\newcommand{\bmu}{\boldsymbol{\mu}}

\newcommand{\halpha}{\hat{\alpha}}

\newcommand{\hbr}{\hat{\br}}

\newcommand{\cR}{\mathcal{R}}

\newcommand{\cD}{\mathcal{D}}

\newcommand{\cC}{\mathcal{C}}

\newcommand{\cH}{\mathcal{H}}

\newcommand{\cS}{\mathcal{S}}

\let\originalleft\left
\let\originalright\right
\renewcommand{\left}{\mathopen{}\mathclose\bgroup\originalleft}
\renewcommand{\right}{\aftergroup\egroup\originalright}
\newcommand{\fronorm}[1]{\left\lVert #1 \right\rVert_{\operatorname{F}}}

\newtheorem{theorem}{Theorem}
\newtheorem{proposition}{Proposition}
\newtheorem{example}{Example}

\begin{document}


\title{\LARGE{Multivariate Rational Approximation of Scattered Data Using the p-AAA Algorithm}
}
    
\author[$\ast$]{Linus Balicki}
\author[$\ast$]{Serkan Gugercin}
\affil[$\ast$]{Department of Mathematics, Virginia Tech, Blacksburg, VA, 24061, USA}
          

\abstract{Many algorithms for approximating data with rational functions are built on interpolation or least-squares approximation. Inspired by the adaptive Antoulas–Anderson (AAA) algorithm for the univariate case, the parametric adaptive Antoulas–Anderson (p-AAA) algorithm extends this idea to the multivariate setting, combining least-squares and interpolation formulations into a single effective approximation procedure. In its original formulation p-AAA operates on grid data, requiring access to function samples at every combination of discrete sampling points in each variable. In this work we extend the p-AAA algorithm to scattered data sets, without requiring uniform/grid sampling.  In other words, our proposed p-AAA formulation operates on a set of arbitrary sampling points and is not restricted to a grid structure for the sampled data. Towards this goal, we introduce several formulations for rational least-squares optimization problems that incorporate interpolation conditions via constraints. We analyze the structure of the resulting optimization problems and introduce structured matrices whose singular value decompositions yield closed-form solutions to the underlying least-squares problems. Several examples illustrate computational aspects and the effectiveness of our proposed procedure. }

\novelty{}

\maketitle

\section{Introduction}
Rational approximation methods for functions of several variables have received increasing attention in recent years; see, e.g., \cite{antoulas2025,antoulas_loewner_2024,boulle2024,hokanson2020,ionita2014,rodriguez2023,benner_survey_2015,deschrijver_robust_2008,feng_posteriori_2017,hund_optimization-based_2022,mlinaric_interpolatory_2024} for a collection of recently proposed approaches. This growing interest is largely based on the tremendous success of single-variable rational approximation algorithms in a wide range of applications \cite{antoulas_chapter_2017,drmac_quadrature-based_2015,hujdurovic_aaa-least_2023,lietaert_automatic_2022,gustavsen_rational_1999,berljafa_rkfit_2017,antoulas_interpolatory_2020,beattie_realization-independent_2012,borghi_mathcal_2024,brennan_contour_2023,guttel_robust_2022,hokanson2017projected,wilber_data-driven_2022,rubin_bounding_2022,nakatsukasa_applications_2025}, many of which are now extended to the multivariate setting. An effective algorithm that falls into this category is the adaptive Antoulas-Anderson (AAA) algorithm \cite{nakatsukasa_aaa_2018}, together with its multivariate formulation, the parametric AAA (p-AAA) algorithm \cite{rodriguez2023}. In the two-variable case, p-AAA aims to approximate a bivariate function $\bff$ with a rational function $\br$ using a given set of sample data
\begin{equation}
    \label{eq:samples}
    \bD = \{ \bff(x,y) \; | \; (x,y) \in \bS \} \subset \C,
\end{equation}
with the corresponding sampling points
\begin{equation}
    \label{eq:samplingPoints}
    \bS = \{ x_1,\ldots,x_N \} \times \{ y_1,\ldots,y_M\} \subset \C^2.
\end{equation}
The key idea of the algorithm is to compute $\br \approx \bff$ such that some of the samples in $\bD$ are interpolated while the remaining data is captured via a least-squares (LS) approximation. This approach leads to an optimization problem at the core of p-AAA: A linear LS problem with equality constraints that enforce interpolation conditions at select points. In its original formulation~\cite{rodriguez2023}, p-AAA assumes that the set $\bS$ forms a grid (i.e., a cartesian product) of sampling points as shown in \eqref{eq:samplingPoints}. This means that for all $i = 1,\ldots,N$ and $j=1,\ldots,M$, it assumes that $\bff(x_i,y_j) \in \bD$ is among the sampled data. Additionally, interpolation in p-AAA is enforced on a sub-grid of interpolation points
\begin{equation}
    \label{eq:interpolationPoints}
    \bI = \{ \lambda_1,\ldots,\lambda_n \} \times \{ \mu_1,\ldots,\mu_m \} \subset \bS
\end{equation}
such that for all $i=1,\ldots,n$ and $j=1,\ldots,m$ we have
\begin{equation*}
\br(\lambda_i,\mu_j) = \bff(\lambda_i,\mu_j).
\end{equation*}
To derive the linear LS problem that is solved in p-AAA, first consider the constrained nonlinear LS problem
\begin{equation}
    \label{eq:introNonlinearLS}
    \min_{\bn,\bd\not\equiv 0} \sum_{(x,y) \in \bS} \left\lvert \bff(x,y) - \frac{\bn(x,y)}{\bd(x,y)} \right\rvert^2 \quad \text{s.t.} \quad \frac{\bn(\lambda,\mu)}{\bd(\lambda,\mu)} = \bff(\lambda,\mu) \quad \text{for} \quad (\lambda,\mu) \in \bI,
\end{equation}
where the rational function $\br$ is represented via
\begin{equation*}
    \br(x,y) = \frac{\bn(x,y)}{\bd(x,y)}.
\end{equation*}
The functions $\bn$ and $\bd$ are the denominator and numerator of a multivariate barycentric form that we will formally introduce in Section~\ref{sec:barycentricForm}. As we will explain in more detail later, the objective function of the optimization problem in \eqref{eq:introNonlinearLS} is generally nonlinear in the optimization variables. As in the univariate case and AAA, the p-AAA algorithm avoids solving the nonlinear LS problem in \eqref{eq:introNonlinearLS} explicitly and instead aims to compute an approximate solution to it by solving a closely related linear LS problem. This optimization problem reads
\begin{equation}
    \label{eq:introPAAALS}
    \min_{\bn,\bd \not\equiv 0} \sum_{(x,y) \in \bS} \left\lvert \bff(x,y) \bd(x,y)- \bn(x,y)\right\rvert^2 \quad \text{s.t.} \quad \frac{\bn(\lambda,\mu)}{\bd(\lambda,\mu)} = \bff(\lambda,\mu) \quad \text{for} \quad (\lambda,\mu) \in \bI.
\end{equation}
In contrast to the nonlinear problem \eqref{eq:introNonlinearLS} the linear LS problem  \eqref{eq:introPAAALS} admits a closed-form solution that can be computed efficiently, as we will discuss in detail later. We note that linear LS problems like the one stated in \eqref{eq:introPAAALS} appear in a similar form in many iterative algorithms that aim to solve the nonlinear LS problem in \eqref{eq:introNonlinearLS}. For example, the SK iteration \cite{sanathanan_transfer_1963}, which has been extended to the two-variable setting in \cite{hokanson2020}, solves a weighted version of the optimization problem in \eqref{eq:introPAAALS} without interpolation constraints in each iteration. In the context of single-variable rational approximation, a much wider range of algorithms based on similar linear LS problems have been proposed \cite{levy_complex-curve_1959,ackermann_second-order_2025,berljafa_rkfit_2017,gustavsen_rational_1999,whitfield_asymptotic_1987}.

In this work, we investigate the setting where the set of sampling points $\bS$ is an arbitrary subset of $\C^2$ and, in particular, is not constrained to follow the grid structure introduced in \eqref{eq:samplingPoints}. In this case, we call $\bS$ a scattered subset of $\C^2$. Our central goal will be to derive a formulation of the p-AAA algorithm that can be applied to such scattered data sets. At the core of this formulation will be a linear LS problem that incorporates constraints for enforcing interpolation conditions on a set of scattered interpolation points. As one of our key contributions, we show how to solve the resulting optimization problem efficiently by exploiting the underlying structure. Towards this goal, we investigate various other types of constrained and unconstrained optimization problems that cover a wide range of multivariate rational approximation settings. We begin by focusing on bivariate rational functions and later extend our results to the case where $\bff$ and $\br$ are $d$-variate functions with $d > 2$.

\begin{figure}[t]
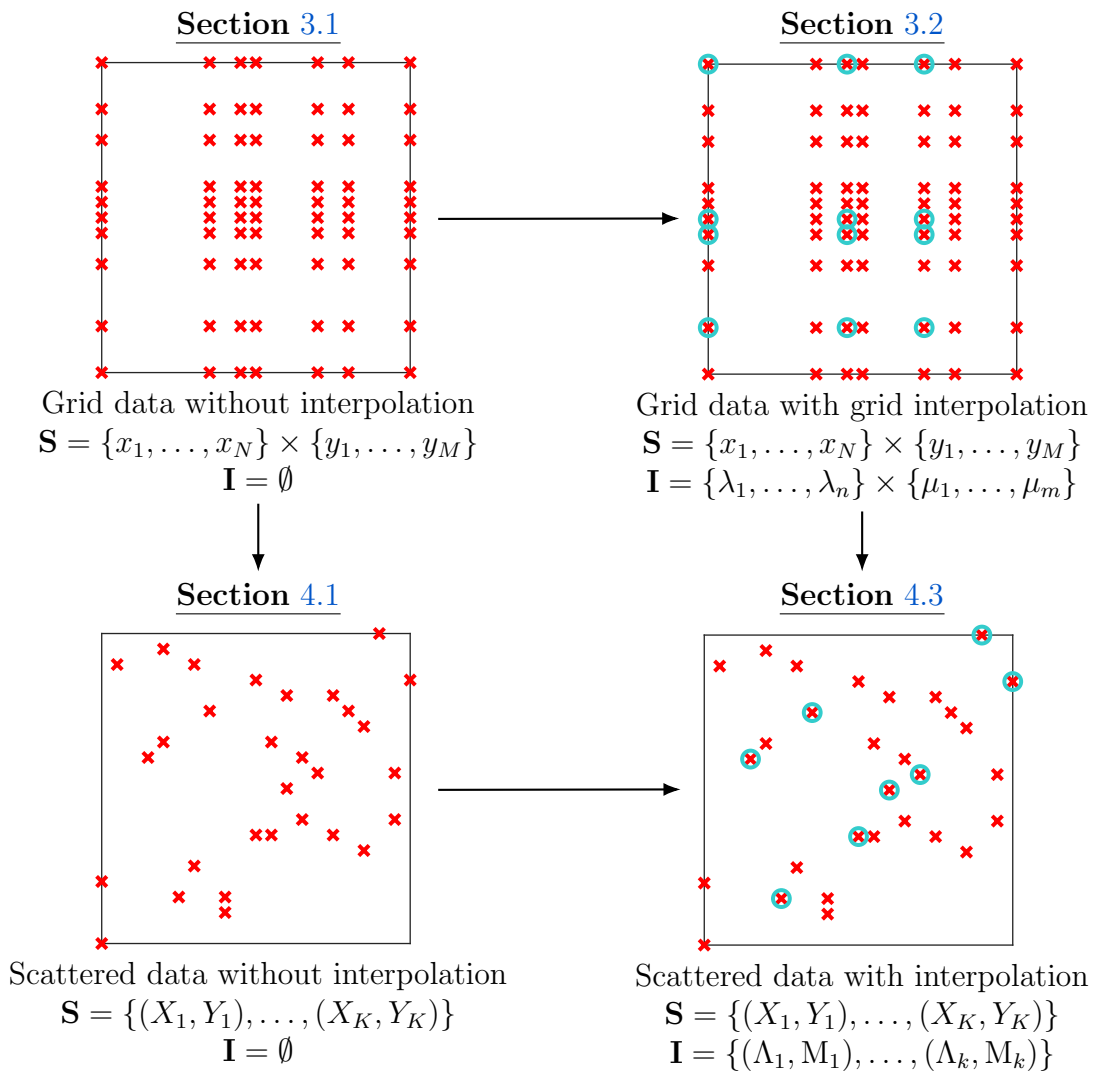

    \include{figures/ls_settings}
    \caption{Various settings for the constrained LS problem~\eqref{eq:introPAAALS} arising in rational approximation that we consider in this manuscript. Red crosses visualize the sampling data set $\bS$ and cyan circles the interpolation set $\bI$. An arrow originating at an illustrated setting indicates that it is a special case of the setting that it is pointing at.}
    \label{fig:overview}
\end{figure}
Our discussions and contributions are presented as follows: First, we revisit bivariate barycentric forms in Section~\ref{sec:barycentricForm}. In Section~\ref{sec:gridApprox} we derive linear LS problems for rational approximation for the case where sampling points form a grid. Based on this discussion we revisit the p-AAA algorithm in Section~\ref{sec:paaa}. In Section~\ref{sec:scatteredApprox}, we investigate linear LS problems that arise in rational approximation with scattered sampling points. This includes our first main result in Section~\ref{sec:scatteredLSInterpolation}, which demonstrates how scattered interpolation conditions can be effectively incorporated into the linear LS problem arising in p-AAA. Section~\ref{sec:scatteredPAAA} contains our second main result given in terms of a formulation of the p-AAA algorithm that operates on scattered sample data. In Section~\ref{sec:numerics} via three numerical examples we demonstrate how our algorithm operates. Finally, we state conclusions and possible future research directions in Section~\ref{sec:conclusions}. Figure~\ref{fig:overview} below provides an overview of the approximation problems in the general form of \eqref{eq:introPAAALS} discussed in Section~\ref{sec:gridApprox} and Section~\ref{sec:scatteredApprox}, along with the corresponding structures of the sample data set $\bS$ and the interpolation set $\bI$. 

\section{Two-variable barycentric forms}
\label{sec:barycentricForm}
The barycentric form of a multivariate rational function is used in several 
approximation and interpolation algorithms \cite{antoulas2025,ionita2014,rodriguez2023}. Its key advantage over various other rational representations is its numerical stability and the fact that interpolation conditions can be enforced in a straight-forward manner. For a rational function in two variables, the barycentric form reads
\begin{equation}
    \label{eq:originalBarycentricForm}
    \br(x,y) = \sum_{i=1}^{n} \sum_{j=1}^{m} \frac{\beta_{ij}}{(x - \lambda_i)(y - \mu_j)} \Bigg/ \sum_{i=1}^{n} \sum_{j=1}^{m} \frac{\alpha_{ij}}{(x - \lambda_i)(y - \mu_j)},
\end{equation}
where
\begin{equation}
    \label{eq:barycentricNodes}
    \blambda = \{ \lambda_1,\ldots,\lambda_n \} \subset \C \quad \text{and} \quad \bmu = \{ \mu_1,\ldots,\mu_m \} \subset \C
\end{equation}
form the set of barycentric nodes via $\blambda \times \bmu$ and the entries of $\alpha,\beta \in \C^{n\times m}$ given by
\begin{equation*}
    \alpha = \begin{bmatrix}
        \alpha_{11} & \cdots & \alpha_{1m} \\ \vdots & & \vdots \\ \alpha_{n1} & \cdots & \alpha_{nm}
    \end{bmatrix} \quad \text{and} \quad \beta = \begin{bmatrix}
        \beta_{11} & \cdots & \beta_{1m} \\ \vdots & & \vdots \\ \beta_{n1} & \cdots & \beta_{nm}
    \end{bmatrix}
\end{equation*}
are called barycentric coefficients. By multiplying the numerator and denominator of $\br$ with the polynomials $(x-\lambda_1) \cdots (x-\lambda_n)$, $(y - \mu_1) \cdots (y - \mu_m)$ and simplifying the resulting fraction we can easily see that $\br$ can be represented as a fraction of degree-$(n-1,m-1)$ polynomials in the variables $x$ and $y$. In this case, we say that $\br$ is a rational function of order-$(n-1,m-1)$. Further, $\br$ has removable singularities at all $(\lambda_i,\mu_j) \in \blambda \times \bmu$. In particular, we have
\begin{equation*}
    \lim_{(x,y) \rightarrow (\lambda_i,\mu_j)} \br(x,y) = \frac{\beta_{ij}}{\alpha_{ij}}~~~~\mbox{if}~~\alpha_{ij} \neq 0.
\end{equation*}
 It is useful to introduce a slightly modified version of the barycentric form that explicitly removes the removable singularities at the barycentric nodes. For this goal, we introduce the basis functions
\begin{equation*}
    g_{\blambda}^{(i)}(x) = \begin{cases}
        \frac{1}{x - \lambda_i} \quad &\text{if } x \notin \blambda, \\
        1 \quad &\text{if } x = \lambda_i, \\
        0 \quad &\text{else}
    \end{cases} \qquad \text{and} \qquad g_{\bmu}^{(j)}(y) = \begin{cases}
        \frac{1}{y - \mu_j} \quad &\text{if } y \notin \bmu, \\
        1 \quad &\text{if } y = \mu_j, \\
        0 \quad &\text{else}
    \end{cases}
\end{equation*}
and the barycentric numerator and denominator as
\begin{equation}
    \label{eq:nd}
    \bn(x,y) = \sum_{i=1}^n \sum_{j=1}^m \beta_{ij} g_{\blambda}^{(i)}(x) g_{\bmu}^{(j)}(y) \quad \text{and} \quad \bd(x,y) = \sum_{i=1}^n \sum_{j=1}^m \alpha_{ij} g_{\blambda}^{(i)}(x) g_{\bmu}^{(j)}(y),
\end{equation}
respectively. This allows for redefining the barycentric form as
\begin{equation}
    \label{eq:barycentricForm}
    \br(x,y) = \frac{\bn(x,y)}{\bd(x,y)}.
\end{equation}
With these definitions, $\bn(\lambda_i,\mu_j)$, $\bd(\lambda_i,\mu_j)$, and 
\begin{equation}
    \label{eq:barycentricNodeEvaluation}
    \br(\lambda_i,\mu_j) = \frac{\beta_{ij}}{\alpha_{ij}}
\end{equation}
are all well-defined, and the singularities at the barycentric nodes are therefore explicitly removed. 

We note that the barycentric nodes $\blambda \times \bmu$ defined via \eqref{eq:barycentricNodes} already appeared as the p-AAA interpolation set $\bI = \blambda \times \bmu$ introduced in \eqref{eq:interpolationPoints}. This means that p-AAA specifically interpolates data at all of the barycentric nodes. It is natural to pick the interpolation points $\bI = \blambda \times \bmu$, as we can choose $\beta_{ij} = \alpha_{ij} \bff(\lambda_i,\mu_j)$ and obtain
\begin{equation*}
    \br(\lambda_i,\mu_j) = \frac{\alpha_{ij} \bff(\lambda_i,\mu_j)}{\alpha_{ij}} = \bff(\lambda_i,\mu_j)
\end{equation*}
for $i=1,\ldots,n$ and $j=1,\ldots,m$ from the relation in \eqref{eq:barycentricNodeEvaluation}. Hence, interpolation on the set $\blambda \times \bmu$ can be enforced by simply imposing the constraint $\beta_{ij} = \alpha_{ij} \bff(\lambda_i,\mu_j)$ on the entries of $\beta$. As we will explain in detail later, p-AAA chooses the nodes via a greedy selection and determines barycentric coefficients via an LS optimization problem. For now, we assume that the set of nodes $\blambda \times \bmu$ is given and we seek to find barycentric coefficients $\alpha$ and $\beta$ such that the rational function $\br$ as in \eqref{eq:barycentricForm} is a good approximation of $\bff$. 

\section{Rational approximation with grid data}
\label{sec:gridApprox}
We begin by considering the set of samples $\bD$ in \eqref{eq:samples}, with a corresponding set of sampling points
\begin{equation*}
    \bS = \{ x_1,\ldots,x_N \} \times \{ y_1,\ldots,y_M \}
\end{equation*}
forming a grid. We note that in this case we can arrange the sample data in $\bD$ as a matrix
\begin{equation*}
    \D = \begin{bmatrix}
        \bff(x_1,y_1) & \cdots & \bff(x_1,y_M) \\
        \vdots & & \vdots \\
        \bff(x_N,y_1) & \cdots & \bff(x_N,y_M)
    \end{bmatrix} \in \C^{N \times M}.
\end{equation*}
The rational approximation problem that we consider here aims to find a rational function $\br = \bn / \bd$ such that $\br \approx \bff$ in some appropriate sense. Typically, approximation quality is measured in terms of the least squares (LS) or maximum error on the sample data $\bD$. Here, we focus on the former, which leads to an optimization problem of the form
\begin{equation}
    \label{eq:rationalLS}
    \min_{\br \in \cR} \sum_{i=1}^N \sum_{j=1}^M \left\lvert \bff(x_i,y_j) - \br(x_i,y_j) \right\rvert^2,
\end{equation}
where $\cR$ is the set of admissible rational functions in barycentric form. We will define $\cR$ for each of the considered approximation settings in their respective section.

\subsection{Linear least-squares for rational approximation}
\label{sec:gridLS}
\label{sec:linearLS}
As our first setting, assume that fixed sets of nodes as in \eqref{eq:barycentricNodes} are given, and we would like to solve \eqref{eq:rationalLS} without enforcing any interpolation conditions. This means we consider the interpolation set
\begin{equation*}
    \bI = \emptyset.
\end{equation*}
In this case the set of admissible rational functions reads
\begin{equation}
    \label{eq:unconstrainedAdmissibleFunctions}
    \cR = \left\{\frac{\bn}{\bd} \; \Big| \; \alpha, \beta \in \C^{n\times m} \; \text{and} \; \fronorm{\left[\alpha,\beta\right]} = 1 \right\},
\end{equation}
where $\bn$ and $\bd$ are as given in \eqref{eq:nd}. In other words, we assume that the basis functions for the barycentric form are fixed and we minimize the LS error over the barycentric coefficients $\alpha$ and $\beta$. Additionally, we impose the normalization constraint $\fronorm{\left[\alpha,\beta\right]} = 1$ to ensure that $\bn$ and $\bd$ are uniquely defined and there are no ambiguous representations of $\br$ arising from scaling the numerator and denominator by a scalar. While there are many other ways to normalize a barycentric form, it will become clear in our discussion below that this particular constraint is especially easy to enforce in the currently discussed setting. Next, we write the objective function in \eqref{eq:rationalLS} as
\begin{equation*}
    \sum_{i=1}^{N} \sum_{j=1}^{M} \left\lvert \bff(x_i,y_j) - \br(x_i,y_j) \right\rvert^2 = \sum_{i=1}^{N} \sum_{j=1}^{M} \frac{1}{\left\lvert \bd(x_i,y_j) \right\rvert^2} \left\lvert \bff(x_i,y_j)\bd(x_i,y_j) - \bn(x_i,y_j)\right\rvert^2.
\end{equation*}
By recalling the structure of $\bd$ introduced in \eqref{eq:nd} we see that the LS problem is nonlinear in the barycentric coefficients $\alpha$ due to the $1 / \lvert \bd(x_i,y_j) \rvert^2$ term. A common approach for approximately solving this nonlinear LS problem is to linearize it by simply dropping the $1 / \lvert \bd(x_i,y_j) \rvert^2$ term \cite{levy_complex-curve_1959}. This yields the linear LS problem
\begin{equation}
\label{eq:baselineGridLinearLS}
    \min_{\bn / \bd \in \cR} \sum_{i=1}^{N} \sum_{j=1}^{M} \left\lvert \bff(x_i,y_j)\bd(x_i,y_j) - \bn(x_i,y_j)\right\rvert^2.
\end{equation}
In order to write this LS problem concisely, we introduce the vectors
\begin{equation*}
    \bx = \left[ x_1, \ldots, x_N \right]^\top \in \C^N \quad \text{and} \quad \by = \left[ y_1, \ldots, y_M \right]^\top \in \C^M
\end{equation*}
based on the sampling points in $\bS$; the matrices $\cC_{\blambda}(\bx) \in \C^{n \times N}$ and $\cC_{\bmu}(\by) \in \C^{m \times M}$ 
\begin{equation*}
    \cC_{\blambda}(\bx) = \begin{bmatrix}
        g_{\blambda}^{(1)}(x_1) & \cdots & g_{\blambda}^{(1)}(x_N) \\ \vdots & & \vdots \\ g_{\blambda}^{(n)}(x_1) & \cdots & g_{\blambda}^{(n)}(x_N) 
    \end{bmatrix} \quad \text{and} \quad \cC_{\bmu}(\by) = \begin{bmatrix}
        g_{\bmu}^{(1)}(y_1) & \cdots & g_{\bmu}^{(1)}(y_M) \\ \vdots & & \vdots \\ g_{\bmu}^{(m)}(y_1) & \cdots & g_{\bmu}^{(m)}(y_M)
    \end{bmatrix};
\end{equation*}
and the diagonal matrix of samples $\cD = \diag(\vectorize(\D)) \in \C^{NM \times NM}$. The following proposition introduces a concise representation for $\eqref{eq:baselineGridLinearLS}$. As we will explain later, \Cref{prop1} is a special case of our main result presented in Theorem~\ref{theorem:scatteredInterpolationLS}, which we prove later.
\begin{proposition} \label{prop1}
    Consider the grid sampling points $\bS$ in $\eqref{eq:samplingPoints}$, the corresponding sample data $\bD$ in \eqref{eq:samples}, fixed sets of nodes $\blambda$ and $\bmu$ as in \eqref{eq:barycentricNodes}, and the set of admissible rational functions $\cR$ in \eqref{eq:unconstrainedAdmissibleFunctions}. Then the optimization problem in \eqref{eq:baselineGridLinearLS} can be written as
    \begin{equation}
    \label{eq:linearLS}
        \min_{\bn / \bd \in \cR} \sum_{i=1}^{N} \sum_{j=1}^{M} \left\lvert \bff(x_i,y_j)\bd(x_i,y_j) - \bn(x_i,y_j)\right\rvert^2 = \min_{\fronorm{\left[ \alpha, \beta \right]} = 1} \left\lVert \A \begin{bmatrix}
            \vectorize(\alpha) \\ \vectorize(\beta)
        \end{bmatrix} \right\rVert_2^2,
    \end{equation}
    where
    \begin{equation*}
        \A = \begin{bmatrix}
            \cD \left( \cC_{\blambda}(\bx) \otimes \cC_{\bmu}(\by) \right)^\top, & -\left( \cC_{\blambda}(\bx) \otimes \cC_{\bmu}(\by) \right)^\top
        \end{bmatrix} \in \C^{NM \times 2nm}.
    \end{equation*}
\end{proposition}
The optimization problem in \eqref{eq:linearLS} has a closed-form solution in terms of the right singular vector of $\A$ associated with its smallest singular value. Computing the LS solution via a singular value decomposition (SVD) will automatically yield barycentric coefficients that satisfy the normalization constraint
\begin{equation*}
    \left\lVert \begin{bmatrix}
        \vectorize(\alpha) \\ \vectorize(\beta)
    \end{bmatrix} \right\rVert_2 = \fronorm{\left[ \alpha, \beta \right]} = 1.
\end{equation*} 
Clearly, $\bn$ and $\bd$ as in \eqref{eq:nd} can be constructed based on the fixed set of nodes $\blambda$ and $\bmu$ in \eqref{eq:barycentricNodes} and the matrices $\alpha$ and $\beta$ that solve the optimization problem in \eqref{eq:linearLS}.

\subsection{Interpolation constraints for grid data}
\label{sec:gridLSInterpolation}
Next, we incorporate grid-based interpolation constraints into the LS problem in \eqref{eq:rationalLS}. In this setting, we again assume that the sampling points form a grid as in \eqref{eq:samplingPoints}. Additionally, we assume that the barycentric nodes form a sub-grid of $\bS$ such that
\begin{equation*}
    \blambda \times \bmu \subset \bS.
\end{equation*}
Further, we introduce an interpolation set that is formed by the barycentric nodes via
\begin{equation*}
    \bI = \blambda \times \bmu.
\end{equation*}
This setting allows for introducing a matrix of interpolation values via
\begin{equation*}
    \H = \begin{bmatrix}
        \bff(\lambda_1,\mu_1) & \cdots & \bff(\lambda_1,\mu_m) \\
        \vdots & & \vdots \\
        \bff(\lambda_n,\mu_1) & \cdots & \bff(\lambda_n,\mu_m)
    \end{bmatrix} \in \C^{n \times m}.
\end{equation*}
Our goal is now to derive a linear LS problem along the lines of \eqref{eq:baselineGridLinearLS} with the additional constraint that
\begin{equation*}
    \br(\lambda,\mu) = \bff(\lambda,\mu) \quad \text{for} \quad (\lambda,\mu) \in \bI.
\end{equation*}
Based on the property of barycentric forms introduced in \eqref{eq:barycentricNodeEvaluation}, we can achieve this goal by setting $\beta_{ij} = \alpha_{ij} \bff(\lambda_i,\lambda_j)$ for $i = 1,\ldots,n$ and $j=1,\ldots,m$. We can write this concisely in matrix-form via
\begin{equation}
\label{eq:betaConstraint}
\beta = \alpha \circ \H,
\end{equation}
where $\circ$ denotes the Hadamard (element-wise) product. This leads to the following set of admissible rational functions with grid interpolation constraints for the currently discussed optimization problem:
\begin{equation}
    \label{eq:gridInterpolationAdmissibleFunctions}
    \cR_{GI} = \left\{ \frac{\bn}{\bd} \; \Big| \; \alpha \in \C^{n\times m}, \; \beta = \alpha \circ \H \; \text{and} \; \fronorm{\alpha} = 1 \right\}.
\end{equation}
Next, we introduce the diagonal matrix of interpolated values $\cH = \diag(\vectorize(\H)) \in \C^{nm \times nm}$. \Cref{prop:2} introduces a concise form for the optimization problem
\begin{equation}
\label{eq:interpolatoryGridLinearLS}
    \min_{\bn / \bd \in \cR_{GI}} \sum_{i=1}^{N} \sum_{j=1}^{M} \left\lvert \bff(x_i,y_j)\bd(x_i,y_j) - \bn(x_i,y_j)\right\rvert^2
\end{equation}
based on the set of admissible rational functions introduced in \eqref{eq:gridInterpolationAdmissibleFunctions}. Again, this is a special case of our main result stated in Theorem~\ref{theorem:scatteredInterpolationLS}. Therefore, we delay its proof. 
\begin{proposition} \label{prop:2}
    Consider the grid sampling points $\bS$ in $\eqref{eq:samplingPoints}$, the corresponding sample data $\bD$ in \eqref{eq:samples}, fixed sets of nodes $\blambda$ and $\bmu$ as in \eqref{eq:barycentricNodes}, and the set of admissible rational functions $\cR_{GI}$ in \eqref{eq:gridInterpolationAdmissibleFunctions}. Then the optimization problem in \eqref{eq:interpolatoryGridLinearLS} can be written as
    \begin{equation}
        \label{eq:interpolatoryLinearLS}
        \min_{\bn / \bd \in \cR_{GI}} \sum_{i=1}^{N} \sum_{j=1}^{M} \left\lvert \bff(x_i,y_j)\bd(x_i,y_j) - \bn(x_i,y_j)\right\rvert^2 = \min_{\fronorm{\alpha} = 1} \left\lVert \L_2 \vectorize(\alpha) \right\rVert_2^2.
    \end{equation}
    where the 2D Loewner matrix reads
    \begin{equation}
    \label{eq:interpolationLSMatrix}
        \L_2 = \cD \left( \cC_{\blambda}(\bx) \otimes \cC_{\bmu}(\by) \right)^\top -\left( \cC_{\blambda}(\bx) \otimes \cC_{\bmu}(\by) \right)^\top \cH \in \C^{NM \times nm}.
    \end{equation}
\end{proposition}
Similar to the previous setting, the optimization problem is solved by $\vectorize(\alpha)$ chosen as the right singular vector of $\L_2$ associated with its smallest singular value. Unlike before, $\beta$ is not directly minimized in the LS problem, as it is fully determined by $\alpha$ through the constraint $\beta = \alpha \circ \H$. Hence, to construct our rational approximant we solve \eqref{eq:interpolatoryLinearLS} for $\alpha$, compute $\beta$ via the relation in \eqref{eq:betaConstraint}, and construct $\bn$ and $\bd$ along the lines of \eqref{eq:nd}. An important question which we have not addressed so far is: How do we choose the barycentric nodes $\blambda \times \bmu$? In the currently discussed setting this is equivalent to the question: Which samples do we interpolate? This is an important question, because the choice of barycentric nodes greatly effects the approximation quality. The p-AAA algorithm, which we revisit in the next subsection, combines the LS problem in \eqref{eq:interpolatoryLinearLS} with an effective strategy for choosing the barycentric nodes.

\subsection{The p-AAA algorithm}
\label{sec:paaa}
The p-AAA algorithm is an iterative algorithm for rational approximation that successively increases the set of interpolation points $\bI$ and solves the LS problem in \eqref{eq:interpolatoryLinearLS}. Importantly, the set of interpolation points in p-AAA coincides exactly with the barycentric nodes in each step, i.e., $\bI = \blambda \times \bmu$. The algorithm is initialized with empty sets $\blambda = \bmu = \emptyset$ and a constant approximant $\br \equiv \operatorname{mean}(\bD)$. The approximant $\br$ is then updated via the two following iteration steps:
\begin{enumerate}
    \item A greedy selection that adds new barycentric nodes.
    \item An update of the barycentric coefficients $\alpha$ via the solution of the linear LS problem in \eqref{eq:interpolatoryLinearLS}.
\end{enumerate}
For the greedy selection we determine the sampling point $(\lambda_*,\mu_*)$ where the current approximation error over all samples is largest, i.e.,
\begin{equation}
    \label{eq:gridGreedy}
    (\lambda_*,\mu_*) = \argmax_{(x,y) \in \bS} \lvert \bff(x,y) - \br(x,y) \rvert.
\end{equation}
The value $\lambda_*$ is then added to $\blambda$ and similarly $\mu_*$ is added to $\bmu$. In the second step we solve the constrained LS problem introduced in \eqref{eq:interpolatoryLinearLS} for the barycentric coefficients $\alpha$ and update $\beta$ via \eqref{eq:betaConstraint}. Based on $\alpha,\beta,\blambda$ and $\bmu$ we can then construct $\bn$ and $\bd$ via \eqref{eq:nd} and finally the rational approximant as $\br = \bn / \bd$. The procedure is summarized in Algorithm~\ref{alg:paaa}. 
\begin{algorithm}[t]
    \caption{p-AAA}\label{alg:paaa}
    \begin{algorithmic}[1]
        \Require{Grid of sampling points $\bS$, sample data $\bD$}
        \Ensure{$\br \approx \bff$}
        \State{Initialize $\br \equiv \operatorname{mean}(\bD)$}
        \State{$\blambda,\bmu \gets \emptyset$}
        \While{error $>$ desired tolerance}
        \State{Determine $ \left( \lambda_*, \mu_* \right) $ via the greedy selection in \eqref{eq:gridGreedy}}
        \State Update the interpolation sets:
        \State\hspace{\algorithmicindent} $\blambda \gets \blambda \cup \{ \lambda_* \}$
        \State\hspace{\algorithmicindent} $\bmu \gets \bmu \cup \{ \mu_* \}$
        \State{Solve the LS problem in \eqref{eq:interpolatoryLinearLS} for $\alpha$ and determine $\beta$ via \eqref{eq:betaConstraint}}
        \State{Update the rational approximant $\br$ using $\alpha, \beta , \blambda$ and $\bmu$}
        \State{error $ \gets \max_{i,j} \; \left\lvert \bff(x_i,y_j) - \br\left(x_i,y_j\right) \right\vert / \max_{i,j} \; \left\lvert \bff(x_i,y_j) \right\rvert $}
        \EndWhile
    \end{algorithmic}
\end{algorithm}
\begin{example}
\label{ex:paaaPeaks}
In Figure~\ref{fig:paaaPeaks} we depict a p-AAA approximation of the MATLAB ``peaks'' function
\begin{equation}
    \label{eq:peaks}
    \bff(x,y) = 3(1 - x)^2 e^{-x^2 - (y + 1)^2}
- 10\left(\frac{x}{5} - x^3 - y^5\right) e^{-x^2 - y^2}
- \frac{1}{3} e^{-(x + 1)^2 - y^2}.
\end{equation}
We use a uniform grid with $40$ sampling points in $[-3,3]$ in each variable. Algorithm~\ref{alg:paaa} was executed until a relative maximum error of $10^{-8}$ was reached. This resulted in an order $(16,16)$ rational approximation after $23$ iterations. Out of the $40^2 = 1600$ sampling points, p-AAA chose to interpolate $17^2 = 289$. Note that in p-AAA the interpolation points always form a sub-grid of the sampling points.
\begin{figure}[t]
\includegraphics{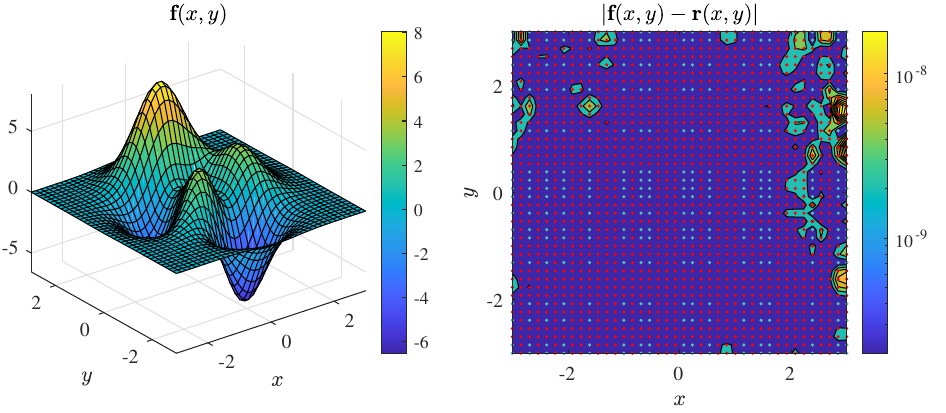}
\caption{Algorithm~\ref{alg:paaa} applied to the peaks function which is depicted in the left subfigure. On the right, we see the p-AAA approximation error after $23$ iterations. Additionally, locations of sampling points (red dots) and interpolation points selected by the algorithm (cyan dots) are depicted.}
\label{fig:paaaPeaks}
\end{figure}
\end{example}

\section{Rational approximation with scattered data}
\label{sec:scatteredApprox}
So far, we have considered sampling points that form a grid. However, grid data may not always be available in practice. This is especially true when we work with more than two variables (as we will discuss in Section~\ref{sec:multiplevariables}), in which case grid sampling may require a substantial amount of computation and memory. If only scattered data is available, the previously introduced version of the p-AAA algorithm is not applicable. Motivated by this limitation, we now generalize our previous discussion to scattered sampling data as outlined in Figure~\ref{fig:overview}. In this setting, we consider the samples in \eqref{eq:samples}, with a corresponding set of scattered sampling points
\begin{equation}
    \label{eq:scatteredSamplingPoints}
    \bS = \{ (X_1,Y_1),\ldots,(X_K,Y_K) \} \subset \C^2.
\end{equation}
We specifically allow $X_i = X_j$ or $Y_i = Y_j$ for $i\neq j$, but do not require it. Clearly, this representation of $\bS$ generalizes the grid formulation that we previously considered in \eqref{eq:samplingPoints}. Similar to the previous section, we consider here an optimization problem of the form
\begin{equation}
    \label{eq:scatteredRationalLS}
    \min_{\br \in \cR} \sum_{i=1}^K \left\lvert \bff(X_i,Y_i) - \br(X_i,Y_i) \right\rvert^2.
\end{equation}
As before, we work with rational functions that are represented in the barycentric form that we introduced in \eqref{eq:barycentricForm}.

\subsection{Linear least-squares for scattered data}
\label{sec:scatteredLS}
We begin with the case in which no interpolation conditions are enforced, that is,
\begin{equation*}
    \bI = \emptyset.
\end{equation*}
As in the setting where the sample data forms a grid, we use the set of admissible rational functions
\begin{equation*}
    \cR = \left\{\frac{\bn}{\bd} \; \Big| \; \alpha, \beta \in \C^{n\times m} \; \text{and} \; \fronorm{\left[\alpha,\beta\right]} = 1 \right\},
\end{equation*}
for the optimization problem in \eqref{eq:scatteredRationalLS}. A linear LS problem along the lines of \eqref{eq:baselineGridLinearLS} adapted to the scattered data setting reads
\begin{equation}
\label{eq:baselineScatteredLinearLS}
    \min_{\bn / \bd \in \cR} \sum_{i=1}^{K} \left\lvert \bff(X_i,Y_i)\bd(X_i,Y_i) - \bn(X_i,Y_i)\right\rvert^2.
\end{equation}
Next, we introduce the vectors
\begin{equation*}
    \bX = \left[ X_1, \ldots, X_K
    \right]^\top  \in \C^K \quad \text{and} \quad \bY = \left[ Y_1,\ldots,Y_K \right]^\top  \in \C^K,
\end{equation*}
the diagonal matrix of samples $\cD = \diag(\bff(X_1,Y_1),\ldots,\bff(X_K,Y_K)) \in \C^{K \times K}$ and denote the Khatri-Rao (column-wise Kronecker) product \cite{kolda_tensor_2009} via the $\odot$-symbol. As in our previous discussions, the following proposition presents a special case of Theorem~\ref{theorem:scatteredInterpolationLS}; thus we delay its proof until then. 
\begin{proposition}
    Consider the scattered sampling points $\bS$ in $\eqref{eq:scatteredSamplingPoints}$, the corresponding sample data $\bD$ in \eqref{eq:samples}, fixed sets of nodes $\blambda$ and $\bmu$ as in \eqref{eq:barycentricNodes}, and the set of admissible rational functions $\cR$ in \eqref{eq:unconstrainedAdmissibleFunctions}. Then the optimization problem in \eqref{eq:baselineScatteredLinearLS} can be written as
    \begin{equation}
    \label{eq:scatteredLinearLS}
        \min_{\bn / \bd \in \cR} \sum_{i=1}^{K} \left\lvert \bff(X_i,Y_i)\bd(X_i,Y_i) - \bn(X_i,Y_i)\right\rvert^2 = \min_{\fronorm{\left[ \alpha, \beta \right]} = 1} \left\lVert \B \begin{bmatrix}
            \vectorize(\alpha) \\ \vectorize(\beta)
        \end{bmatrix} \right\rVert_2^2
    \end{equation}
    where
    \begin{equation}
        \label{eq:scatteredLSMatrix}
        \B = \begin{bmatrix}
            \cD \left( \cC_{\blambda}(\bX) \odot \cC_{\bmu}(\bY) \right)^\top, & -\left( \cC_{\blambda}(\bX) \odot \cC_{\bmu}(\bY) \right)^\top
        \end{bmatrix} \in \C^{K \times 2nm}.
    \end{equation}
\end{proposition}
It can easily be verified that the LS problem in \eqref{eq:scatteredLinearLS} exactly recovers \eqref{eq:linearLS}, if the sampling points in $\bS$ form a grid as in \eqref{eq:samplingPoints}. This means that we have introduced a generalization of the problem discussed in Section~\ref{sec:gridLS}.

\subsection{Singularities of sparse barycentric forms}
\label{sec:barycentricSingularities}
Next, we consider scattered sampling points $\bS$ as in \eqref{eq:scatteredSamplingPoints} and interpolation points given by the scattered set
\begin{equation*}
    \bI = \{ (\Lambda_1,\Mu_1),\ldots,(\Lambda_k,\Mu_k) \}.
\end{equation*}
Before presenting our main result in the next section, we first highlight a potential pitfall that arises in the scattered interpolation setting. Consider a barycentric form that reads
\begin{equation}
\label{eq:sparseBarycentricForm}
    \hbr(x,y) = \sum_{i=1}^{K} \frac{\halpha_i \bff(\Lambda_i,\Mu_i)}{(x - \Lambda_i)(y - \Mu_i)} \Bigg/ \sum_{i=1}^{K} \frac{\halpha_{i}}{(x - \Lambda_i)(y - \Mu_i)}.
\end{equation}
This barycentric form is different than the one introduced in \eqref{eq:originalBarycentricForm} as its barycentric nodes are in general given by the scattered set $\bI$ rather than the grid $\blambda \times \bmu$. It can easily be verified that $\hbr$ is a rational function of order $(K-1,K-1)$ and has removable singularities at $(\Lambda_i,\Mu_i)$ that evaluate to $\bff(\Lambda_i,\Mu_i)$. Hence, we can use $\hbr$ to enforce interpolation conditions on a scattered set. However, $\hbr$ typically has other singularities at undesirable locations. For example, assume that the interpolation points and samples are distinct. This means for all $i \neq j$ we assume $\Lambda_i \neq \Lambda_j$, $\Mu_i \neq \Mu_j$ and $\bff(\Lambda_i,\Mu_i) \neq \bff(\Lambda_j,\Mu_j)$. Based on the definition of $\hbr$ in \eqref{eq:sparseBarycentricForm} we obtain the limits
\begin{equation*}
    \lim_{x \rightarrow \Lambda_i} \hbr(x,y) = \bff(\Lambda_i,\Mu_i) \quad \text{and} \quad \lim_{y \rightarrow \Mu_j} \hbr(x,y) = \bff(\Lambda_j,\Mu_j).
\end{equation*}
From this it follows that $\hbr$ has a non-removable singularity at $(\Lambda_i,\Mu_j)$ for all $i \neq j$ because
\begin{equation*}
    \lim_{y \rightarrow \Mu_j} \lim_{x \rightarrow \Lambda_i} \hbr(x,y) = \bff(\Lambda_i,\Mu_i) \neq \bff(\Lambda_j,\Mu_j) = \lim_{x \rightarrow \Lambda_i} \lim_{y \rightarrow \Mu_j} \hbr(x,y).
\end{equation*}
In particular, it is possible to rewrite $\hbr$ in \eqref{eq:sparseBarycentricForm} exactly in the barycentric form introduced in \eqref{eq:originalBarycentricForm}, but now with barycentric coefficients that satisfy $\alpha_{ij} = \beta_{ij} = 0$ for $i \neq j$. Thus the singularities are guaranteed to appear at the combinations of interpolation nodes where we have no data available. For example, $\hbr$ interpolates data at $(\Lambda_i, \Mu_i)$ and $(\Lambda_j, \Mu_j)$ but has a singularity at $(\Lambda_i, \Mu_j)$. In most applications, this is undesirable behavior. We will avoid these issues by working with the barycentric form introduced in \eqref{eq:barycentricForm} instead. With this barycentric form we will generically obtain $\alpha_{ij},\beta_{ij} \neq 0$, and hence no singularities at $(\Lambda_i, \Mu_j)$.
\begin{example}
\label{ex:singularities}
Consider the interpolation points
\begin{equation*}
    (\Lambda_1,\Mu_1) = (-1,1), \quad
    (\Lambda_2,\Mu_2) = (3,5)
\end{equation*}
with the corresponding interpolation samples
\begin{equation*}
    \bff(\Lambda_1,\Mu_1) = -1, \quad
    \bff(\Lambda_2,\Mu_2) = 3,
\end{equation*}
and barycentric coefficients
\begin{equation*}
    \halpha_1 = -5, \quad
    \halpha_2 = 1.
\end{equation*}
Now consider the interpolating rational function
\begin{equation*}
    \hbr(x,y) = \frac{
        \frac{-5}{(x + 1)(y - 1)} + \frac{3}{(x - 3)(y - 5)}
    }{
        \frac{5}{(x + 1)(y - 1)} + \frac{1}{(x - 3)(y - 5)}
    }.
\end{equation*}
At $(\Lambda_1,\Mu_1) = (-1,1)$ and $(\Lambda_2,\Mu_2) = (3,5)$ the function interpolates $\bff$. However, $\hbr$ has singularities at $(\Lambda_1,\Mu_2) = (-1,5)$ and $(\Lambda_2,\Mu_1)=(3,1)$. These singularities are visualized in Figure~\ref{fig:singularities}. Having singularities within the approximation domain is typically undesirable. This is guaranteed to happen with the sparse barycentric form $\hbr$ in this example, for any choice of $\halpha_1 \neq 0$ and $\halpha_2 \neq 0$.

\begin{figure}
\centering
\includegraphics{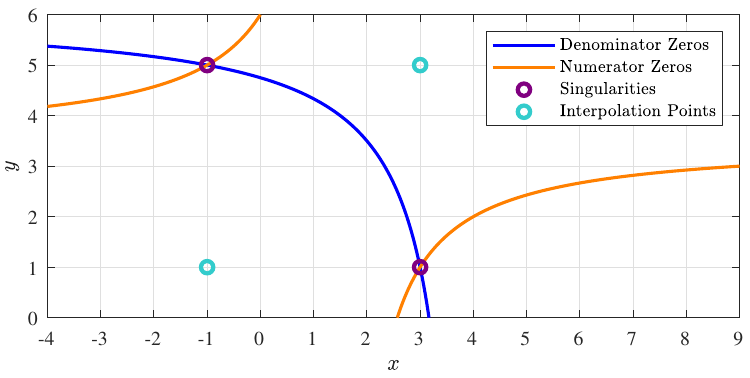}
\caption{Zeros of the numerator and denominator, singularities, and interpolation points discussed in Example~\ref{ex:singularities}.}
\label{fig:singularities}
\end{figure}
\end{example}

\subsection{Interpolation constraints for scattered data}
\label{sec:scatteredLSInterpolation}
For our final setting we consider scattered sampling points given by $\bS$ in \eqref{eq:scatteredSamplingPoints}, barycentric nodes $\blambda \times \bmu$ as in \eqref{eq:barycentricNodes}, and the scattered interpolation set
\begin{equation}
    \label{eq:scatteredInterpolationPoints}
    \bI = \{ (\Lambda_1,\Mu_1),\ldots,(\Lambda_k,\Mu_k) \} \subset \blambda \times \bmu.
\end{equation}
Our key idea for enforcing interpolation on the scattered set $\bI$ via a rational function in barycentric form as in \eqref{eq:barycentricForm} is to constrain some of the entries of the coefficient matrix $\beta$ and leave the remaining ones as optimization variables in the LS problem. Specifically, we enforce the constraint
\begin{equation}
    \label{eq:scatteredInterpolationConstraint}
    \beta_{ij} = \alpha_{ij} \bff(\lambda_i,\mu_j) \quad \text{if} \quad (\lambda_i,\mu_j) \in \bI,
\end{equation}
which is equivalent to the interpolation condition
\begin{equation}
    \label{eq:scatterdInterpolationCondition}
    \br(\Lambda_i,\Mu_i) = \bff(\Lambda_i,\Mu_i)
\end{equation}
for $i=1,\ldots,k$ assuming $\alpha_{ij} \neq 0$. We point out that $\beta$ is a matrix with $nm$ entries, and the condition in \eqref{eq:scatteredInterpolationConstraint} constrains exactly $k$ of these $nm$ entries in order to enforce interpolation on the $k$ points in $\bI$. In order to clearly separate the constrained and unconstrained entries of $\beta$, we introduce the matrices $\cS_{\bc} \in \R^{k \times nm}$ and $\cS_{\bu} \in \R^{(nm - k) \times nm}$ that allow for extracting the $k$ constrained and $nm - k$ unconstrained entries of $\beta$, respectively. A precise definition of these matrices is provided in Appendix~\ref{sec:selectiontwovar}. The vector of the $k$ constrained entries of $\beta$ reads
\begin{equation*}
    \beta_{\bc} = \cS_{\bc} \vectorize(\beta) \in \C^{k}
\end{equation*}
while the $nm-k$ unconstrained entries are arranged in the vector
\begin{equation*}
    \beta_{\bu} = \cS_{\bu} \vectorize(\beta) \in \C^{nm - k}.
\end{equation*}
Hence, the matrices $\cS_{\bc}$ and $\cS_{\bu}$ serve as ``indexing'' or ``selection'' operators that allow for extracting the constrained and unconstrained entries of $\vectorize(\beta)$. We can reconstruct $\beta$ via the relation
\begin{equation}
    \label{eq:betaReconstruction}
    \vectorize(\beta) = \cS_{\bc}^\top \beta_{\bc} + \cS_{\bu}^\top \beta_{\bu}.
\end{equation}
Based on these definitions we define our admissible set of rational functions with scattered interpolation constraints as
\begin{equation}
    \label{eq:scatteredAdmissibleFunctions}
    \cR_{SI} = \left\{ \frac{\bn}{\bd} \; \Big| \; \alpha \in \C^{n\times m}, \; \beta_{\bu} \in \C^{nm - k}, \; \beta_{\bc} = \alpha_{\bc} \circ \bH \; \text{and} \; \fronorm{\alpha}^2 + \lVert\beta_{\bu}
   \rVert^2 = 1  \right\},
\end{equation}
where $\bH = \left[ \bff(\Lambda_1,\Mu_1),\ldots,\bff(\Lambda_k,\Mu_k) \right]^\top \in \C^{k}$ and $\alpha_{\bc} = \cS_{\bc} \vectorize(\alpha)$. Further, we introduce the diagonal matrix $\cH = \diag(\cS_{\bc}^\top \bH) \in \C^{nm \times nm}$ which leads to our main result. 
\begin{theorem}
    \label{theorem:scatteredInterpolationLS}
    Consider the scattered sampling points $\bS$ in $\eqref{eq:scatteredSamplingPoints}$, the corresponding sample data $\bD$ in \eqref{eq:samples}, fixed sets of nodes $\blambda$ and $\bmu$ as in \eqref{eq:barycentricNodes}, the interpolation set $\bI$ in \eqref{eq:scatteredInterpolationPoints}, and the set of admissible rational functions $\cR_{SI}$ in \eqref{eq:scatteredAdmissibleFunctions}. Then we can write
    \begin{equation}
    \label{eq:scatteredInterpolationLS}
    \min_{\bn / \bd \in \cR_{SI}} \sum_{i=1}^{K} \left\lvert \bff(X_i,Y_i)\bd(X_i,Y_i) - \bn(X_i,Y_i)\right\rvert^2 = \min_{\fronorm{\alpha}^2 + \lVert\beta_{\bu}
   \rVert^2 = 1 } \left\lVert \M \begin{bmatrix}
        \vectorize(\alpha) \\ \beta_{\bu}
    \end{bmatrix} \right\rVert_2^2,
\end{equation}
where $\M \in \C^{K \times (2nm - k)}$ is given by
\begin{equation}
    \label{eq:scatteredInterpolationLSMatrix}
    \M = \begin{bmatrix}
        \cD \left( \cC_{\blambda}(\bX) \odot \cC_{\bmu}(\bY) \right)^\top - \left( \cC_{\blambda}(\bX) \odot \cC_{\bmu}(\bY) \right)^\top \cH, & -\left( \cC_{\blambda}(\bX) \odot \cC_{\bmu}(\bY) \right)^\top \cS_{\bu}^\top
    \end{bmatrix}.
\end{equation}
\end{theorem}
\begin{proof}
    First, note that we can write the barycentric numerator and denominator introduced in \eqref{eq:nd} as
    \begin{equation*}
        \bd(X_i,Y_i) = \left[\cC_{\blambda}(X_i) \otimes \cC_{\bmu}(Y_i)\right]^\top \vectorize(\alpha) \quad \text{and} \quad \bn(X_i,Y_i) = \left[\cC_{\blambda}(X_i) \otimes \cC_{\bmu}(Y_i)\right]^\top \vectorize(\beta).
    \end{equation*}
    Plugging these expressions into the objective function of the optimization problem on the left-hand side of \eqref{eq:scatteredInterpolationLS} gives
    \begin{align}
        \sum_{i=1}^{K} &\left\lvert \bff(X_i,Y_i) \left[\cC_{\blambda}(X_i) \otimes \cC_{\bmu}(Y_i)\right]^\top \vectorize(\alpha) - \left[\cC_{\blambda}(X_i) \otimes \cC_{\bmu}(Y_i)\right]^\top \vectorize(\beta) \right\rvert^2 \nonumber \\
        & = \left\lVert \cD \left[ \cC_{\blambda}(\bX) \odot \cC_{\bmu}(\bY) \right]^\top \vectorize(\alpha) - \left[ \cC_{\blambda}(\bX) \odot \cC_{\bmu}(\bY) \right]^\top \vectorize(\beta) \right\rVert_2^2. \label{eq:proofLSFormula}
    \end{align}
    Next, we incorporate the interpolation constraint $\beta_{\bc} = \alpha_{\bc} \circ \bH$ by writing
    \begin{equation*}
        \vectorize(\beta) = \cS_{\bc}^\top \beta_{\bc} + \cS_{\bu}^\top \beta_{\bu} = \cS_{\bc}^\top (\alpha_{\bc} \circ \bH) + \cS_{\bu}^\top \beta_{\bu}.
    \end{equation*}
    Then using the fact that
    \begin{equation*}
        \cS_{\bc}^\top (\alpha_{\bc} \circ \bH) = \cS_{\bc}^\top \diag(\bH) \alpha_{\bc} = \diag(\cS_{\bc}^\top \bH) \vectorize(\alpha)
    \end{equation*}
    and recalling that we defined $\cH = \diag(\cS_{\bc}^\top \bH)$, we obtain
    \begin{equation*}
        \vectorize(\beta) = \cH \vectorize(\alpha) + \cS_{\bu}^\top \beta_{\bu}.
    \end{equation*}
    Plugging this expression into \eqref{eq:proofLSFormula} gives
    \begin{align*}
        &\left\lVert \cD \left[ \cC_{\blambda}(\bX) \odot \cC_{\bmu}(\bY) \right]^\top \vectorize(\alpha) - \left[ \cC_{\blambda}(\bX) \odot \cC_{\bmu}(\bY) \right]^\top \left[ \cH \vectorize(\alpha) + \cS_{\bu}^\top \beta_{\bu} \right] \right\rVert_2^2 \\
        =& \left\lVert \left( \cD \left[ \cC_{\blambda}(\bX) \odot \cC_{\bmu}(\bY) \right]^\top - \left[ \cC_{\blambda}(\bX) \odot \cC_{\bmu}(\bY) \right]^\top \cH \right) \vectorize(\alpha) - \left[ \cC_{\blambda}(\bX) \odot \cC_{\bmu}(\bY) \right]^\top \cS_{\bu}^\top \beta_{\bu} \right\rVert_2^2 \\
        =& \left\lVert \begin{bmatrix}
        \cD \left( \cC_{\blambda}(\bX) \odot \cC_{\bmu}(\bY) \right)^\top - \left( \cC_{\blambda}(\bX) \odot \cC_{\bmu}(\bY) \right)^\top \cH, & -\left( \cC_{\blambda}(\bX) \odot \cC_{\bmu}(\bY) \right)^\top \cS_{\bu}^\top
        \end{bmatrix} \begin{bmatrix}
        \vectorize(\alpha) \\ \beta_{\bu}
        \end{bmatrix}\right\rVert_2^2 \\
        =& \left\lVert \M \begin{bmatrix}
        \vectorize(\alpha) \\ \beta_{\bu}
        \end{bmatrix} \right\rVert_2^2.
    \end{align*}
\end{proof}
As with all previously discussed linear LS problems, \eqref{eq:scatteredInterpolationLS} admits a closed form solution in terms of the right singular vector of $\M$ associated with its smallest singular value. Note that this solution will automatically satisfy
\begin{equation*}
    \left\lVert \begin{bmatrix}
        \vectorize(\alpha) \\ \beta_{\bu}
    \end{bmatrix} \right\rVert_2^2 =  \fronorm{\alpha}^2 + \lVert\beta_{\bu}
   \rVert^2 = 1,
\end{equation*}
which motivates our previously introduced normalization constraint.

We emphasize that the optimization problem in \eqref{eq:scatteredInterpolationLS} is a generalization of all previously derived LS problems. To recover the LS problem in \eqref{eq:scatteredLinearLS} (and similarly \eqref{eq:linearLS}) we assume that no interpolation conditions are enforced. In this case, the matrix $\cS_{\bu}$ will be the identity matrix and $\cH$ will ``disappear'', resulting in the LS matrix $\M$ in \eqref{eq:scatteredInterpolationLSMatrix} that matches the matrix in \eqref{eq:scatteredLSMatrix}. To recover \eqref{eq:interpolatoryLinearLS} we enforce interpolation on the entire grid $\blambda \times \bmu$. In this case we have that $\cS_{\bu}$ ``disappears'', $\cS_{\bc}$ is the identity matrix and thus $\cH = \diag(\vectorize(\H))$. It is easily verified that $\M$ will coincide with $\L_2$ introduced in \eqref{eq:interpolationLSMatrix} when these conditions are met. After an illustrative example, we investigate how the derived LS problem yields a formulation of the p-AAA algorithm that operates on scattered data sets in the next section.
\begin{example}
Here, we explicitly calculate the previously introduced quantities for a small example. Consider the function
\begin{equation*}
    \bff(x,y) = \frac{x^2 + xy + y + 1}{x + y + 5}.
\end{equation*}
We choose the scattered sampling points
\begin{align*}
(X_1,Y_1)&=(-2,-2), & (X_2,Y_2)&=(-2,1), & (X_3,Y_3)&=(-1,1), \\
(X_4,Y_4)&=(-1,2), & (X_5,Y_5)&=(0,-1), & (X_6,Y_6)&=(0,2), \\
(X_7,Y_7)&=(1,-1), & (X_8,Y_8)&=(2,-2), & (X_9,Y_9)&=(2,2),
\end{align*}
forming the set $\bS$. The associated samples forming the set $\bD$ are given by
\begin{align*}
f(X_1,Y_1)&=7, & f(X_2,Y_2)&=1, & f(X_3,Y_3)&=\tfrac{2}{5}, \\
f(X_4,Y_4)&=\tfrac{1}{3}, & f(X_5,Y_5)&=0, & f(X_6,Y_6)&=\tfrac{3}{7}, \\
f(X_7,Y_7)&=0, & f(X_8,Y_8)&=-\tfrac{1}{5}, & f(X_9,Y_9)&=\tfrac{11}{9}.
\end{align*}
The interpolation points
\begin{equation*}
    (\Lambda_1,\Mu_1) = (-1,2), \qquad (\Lambda_2,\Mu_2) = (1,-1),
\end{equation*}
form the set $\bI$. For the barycentric nodes we choose
\begin{equation*}
    \lambda_1 = -1, \quad \lambda_2 = 1, \quad \mu_1 = -1, \quad \mu_2 = 2,
\end{equation*}
which results in the set of barycentric nodes given by
\begin{equation*}
    \blambda \times \bmu = \left\{ (-1,-1), (-1,2), (1,-1), (1,2) \right\}.
\end{equation*}
All of these quantities are visualized in the left-most and middle subplot of Figure~\ref{fig:simpleCalculation}. Based on the given set of barycentric nodes, sampling, and interpolation data, we would like to approximate $\bff$ with the rational function 
\begin{equation*}
    \br(x,y) =
    \frac{
        \frac{\beta_{11}}{(x+1)(y+1)} + \frac{\beta_{12}}{(x+1)(y-2)} + \frac{\beta_{21}}{(x-1)(y+1)} + \frac{\beta_{22}}{(x-1)(y-2)}
    }{
        \frac{\alpha_{11}}{(x+1)(y+1)} + \frac{\alpha_{12}}{(x+1)(y-2)} + \frac{\alpha_{21}}{(x-1)(y+1)} + \frac{\alpha_{22}}{(x-1)(y-2)}
    }
\end{equation*}
via the optimization problem in \eqref{eq:scatteredInterpolationLS}. In particular, we want $\br$ to interpolate $\bff$ on $\bI$ such that $\br(-1,2) = \bff(-1,2)$ and $\br(1,-1) = \bff(1,-1)$. To construct the matrix $\M$ we first need to put together several intermediate quantities. First, the masking matrices $\cS_{\bc}$ and $\cS_{\bu}$ are given by
\begin{equation*}
    \cS_{\bc} = 
    \begin{bmatrix}
        0 & 1 & 0 & 0 \\
        0 & 0 & 1 & 0
    \end{bmatrix}
    \qquad \text{and} \qquad
    \cS_{\bu} =
    \begin{bmatrix}
        1 & 0 & 0 & 0 \\
        0 & 0 & 0 & 1
    \end{bmatrix}.
\end{equation*}
Hence, the entries of $\beta$ that are constrained by the interpolation condition are given by
\begin{equation*}
    \beta_{\bc} = \begin{bmatrix}
        \beta_{12} \\ \beta_{21}
    \end{bmatrix} = \begin{bmatrix}
        0 & 1 & 0 & 0 \\
        0 & 0 & 1 & 0
    \end{bmatrix} \begin{bmatrix}
        \beta_{11} \\ \beta_{12} \\ \beta_{21} \\ \beta_{22}
    \end{bmatrix} = \cS_{\bc} \vectorize(\beta).
\end{equation*}
Similarly, the entries of $\beta$ that are the unconstrained optimization variables in the LS problem are given by
\begin{equation*}
    \beta_{\bu} = \begin{bmatrix}
        \beta_{11} \\ \beta_{22}
    \end{bmatrix} = \begin{bmatrix}
        1 & 0 & 0 & 0 \\
        0 & 0 & 0 & 1
    \end{bmatrix} \begin{bmatrix}
        \beta_{11} \\ \beta_{12} \\ \beta_{21} \\ \beta_{22}
    \end{bmatrix} = \cS_{\bu} \vectorize(\beta).
\end{equation*}
\begin{figure}
\centering
\includegraphics{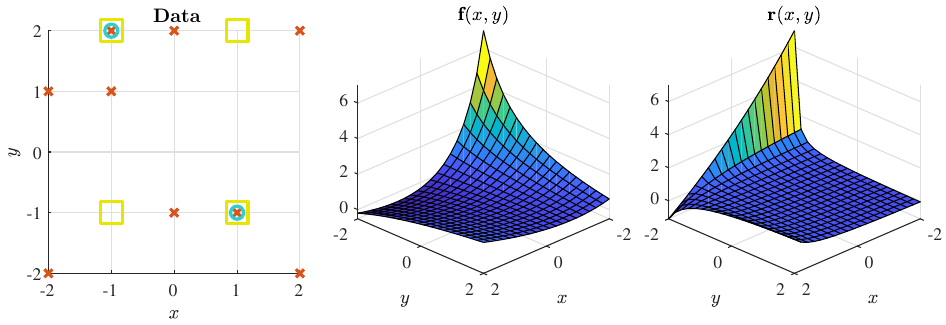}
\caption{Visualization of data and approximation for the calculated example. The left subfigure shows sampling points (red crosses), barycentric nodes (yellow squares) and interpolation points (cyan circles). The function $\bff$ and the approximation $\br$ are shown in the middle and right subfigure, respectively.}
\label{fig:simpleCalculation}
\end{figure}
Next, $\bX$ and $\bY$ are given by
\begin{align*}
    \bX &= \left[ X_1, X_2, X_3, X_4, X_5, X_6, X_7, X_8, X_9 \right]^\top
        = \left[ -2,\ -2,\ -1,\ -1,\ 0,\ 0,\ 1,\ 2,\ 2 \right]^\top, \\
    \bY &= \left[ Y_1, Y_2, Y_3, Y_4, Y_5, Y_6, Y_7, Y_8, Y_9 \right]^\top
        = \left[ -2,\ 1,\ 1,\ 2,\ -1,\ 2,\ -1,\ -2,\ 2 \right]^\top.
\end{align*}
This leads to the matrices
\begin{align*} \renewcommand{\arraystretch}{1.2}
\cC_{\blambda}(\bX) &= 
\begin{bmatrix}
-1 & -1 & 1 & 1 & 1 & 1 & 0 & \tfrac{1}{3} & \tfrac{1}{3} \\
-\tfrac{1}{3} & -\tfrac{1}{3} & 0 & 0 & -1 & -1 & 1 & 1 & 1
\end{bmatrix}, \\
\cC_{\bmu}(\bY) &= 
\begin{bmatrix}
-1 & \tfrac{1}{2} & \tfrac{1}{2} & 0 & 1 & 0 & 1 & -1 & 0 \\
-\tfrac{1}{4} & -1 & -1 & 1 & 0 & 1 & 0 & -\tfrac{1}{4} & 1
\end{bmatrix},
\end{align*}
and
\begin{equation*} \renewcommand{\arraystretch}{1.2}
\cC_{\blambda}(\bX) \odot \cC_{\bmu}(\bY) =
\begin{bmatrix}
1 & -\tfrac{1}{2} & \tfrac{1}{2} & 0 & 1 & 0 & 0 & -\tfrac{1}{3} & 0 \\
\tfrac{1}{4} & 1 & -1 & 1 & 0 & 1 & 0 & -\tfrac{1}{12} & \tfrac{1}{3} \\
\tfrac{1}{3} & -\tfrac{1}{6} & 0 & 0 & -1 & 0 & 1 & -1 & 0 \\
\tfrac{1}{12} & \tfrac{1}{3} & 0 & 0 & 0 & -1 & 0 & -\tfrac{1}{4} & 1
\end{bmatrix}.
\end{equation*}
The matrix $\cD$ is a diagonal matrix formed by the samples
\begin{equation*}
    \cD = \diag(\bff(X_1,Y_1),\bff(X_2,Y_2), \ldots, \bff(X_9,y_9)) = \diag(7,1,\ldots,\tfrac{11}{9})
\end{equation*}
and $\cH$ is given by
\begin{equation*}
    \cH = \diag\left(\cS_{\bc}^\top \begin{bmatrix}
        \bff(\Lambda_1,\Mu_1) \\ \bff(\Lambda_2,\Mu_2)
    \end{bmatrix}\right) = \diag(0, \tfrac{1}{3}, 0, 0).
\end{equation*}
Finally, we can construct the LS matrix in \eqref{eq:scatteredInterpolationLSMatrix} which reads
\begin{equation*}
\renewcommand{\arraystretch}{1.2}
\M = \begin{bmatrix}
7 & \tfrac{5}{3} & \tfrac{7}{3} & \tfrac{7}{12} & -1 & -\tfrac{1}{12} \\
-\tfrac{1}{2} & \tfrac{2}{3} & -\tfrac{1}{6} & \tfrac{1}{3} & \tfrac{1}{2} & -\tfrac{1}{3} \\
\tfrac{1}{5} & -\tfrac{1}{15} & 0 & 0 & -\tfrac{1}{2} & 0 \\
0 & 0 & 0 & 0 & 0 & 0 \\
0 & 0 & 0 & 0 & -1 & 0 \\
0 & \tfrac{2}{21} & 0 & -\tfrac{3}{7} & 0 & 1 \\
0 & 0 & 0 & 0 & 0 & 0 \\
\tfrac{1}{15} & \tfrac{2}{45} & \tfrac{1}{5} & \tfrac{1}{20} & \tfrac{1}{3} & \tfrac{1}{4} \\
0 & \tfrac{8}{27} & 0 & \tfrac{11}{9} & 0 & -1
\end{bmatrix}.
\end{equation*}
Note that $\M$ has two zero rows that arise due to the two interpolation points whose contribution to the LS error is zero. Computing the singular vector corresponding to the smallest singular value of $\M$ gives the following solution to the LS problem
\begin{equation*}
\left[\alpha_{11}, \alpha_{12}, \alpha_{21}, \alpha_{22}, \beta_{11}, \beta_{22}\right]^\top
\approx
\left[-0.3222, 0.0633, 0.9246, -0.1376, -0.0624, -0.1200\right]^\top.
\end{equation*}
Based on this, we construct the constrained $\beta$ coefficients via
\begin{equation*}
    \beta_{\bc} = \begin{bmatrix}
        \beta_{12} \\ \beta_{21}
    \end{bmatrix} = \begin{bmatrix}
        \alpha_{12} \\ \alpha_{21}
    \end{bmatrix} \circ \begin{bmatrix}
        \bff(-1,2) \\ \bff(1,-1)
    \end{bmatrix}\approx \begin{bmatrix}
        0.0633 \\ 0.9246
    \end{bmatrix} \circ \begin{bmatrix}
        \tfrac{1}{3} \\ 0
    \end{bmatrix} = \begin{bmatrix}
         0.0211 \\ 0
    \end{bmatrix}.
\end{equation*}
Hence, the full numerator coefficients and denominator coefficients read
\begin{equation*}
    \begin{bmatrix}
        \beta_{11} \\ \beta_{12} \\ \beta_{21} \\ \beta_{22}
    \end{bmatrix} = \cS_{\bc}^\top \beta_{\bc} + \cS_{\bu}^\top \beta_{\bu} = \begin{bmatrix}
        0 \\ \beta_{12} \\ \beta_{21} \\ 0
    \end{bmatrix} + \begin{bmatrix}
        \beta_{11} \\ 0 \\ 0 \\ \beta_{22}
    \end{bmatrix} \approx \begin{bmatrix}
        -0.0624 \\ 0.0211 \\ 0 \\ -0.1200
    \end{bmatrix}
\end{equation*}
and
\begin{equation*}
    \begin{bmatrix}
        \alpha_{11} \\ \alpha_{12} \\ \alpha_{21} \\ \alpha_{22}
    \end{bmatrix} \approx \begin{bmatrix}
        -0.3222 \\ 0.0633 \\ 0.9246 \\ -0.1376
    \end{bmatrix}.
\end{equation*}
Based on the barycentric coefficients and previously specified barycentric nodes, we can construct the rational function $\br$. For this example $\br$ is depicted in the right-most subplot of Figure~\ref{fig:simpleCalculation}.
\end{example}

\section{p-AAA with scattered data sets}
\label{sec:scatteredPAAA}
In this section we leverage Theorem~\ref{theorem:scatteredInterpolationLS} and extend the original grid-data p-AAA  as given in Algorithm~\ref{alg:paaa} to scattered data sets. We begin with the bivariate case, and extend the discussion to functions that depend on $d > 2$ variables afterwards.

\subsection{Scattered p-AAA for bivariate approximation}

As in the original p-AAA formulation, we follow a greedy procedure for selecting the nodes $\blambda$ and $\bmu$ and solve the LS problem in \eqref{eq:scatteredInterpolationLS} to compute the barycentric coefficients in p-AAA with scattered data. Similar to the original algorithm, the greedy procedure selects new barycentric nodes via
\begin{equation}
    \label{eq:nonGridGreedy}
    (\lambda_*,\mu_*) = \argmax_{(x,y) \in \bS} \lvert \bff(x,y) - \br(x,y) \rvert.
\end{equation}
While the original algorithm enforces interpolation on the grid $\blambda \times \bmu$ (and thus the  barycentric coefficients $\beta$ are completely determined by $\alpha$ via \eqref{eq:betaConstraint}), we have now more flexibility in choosing the set of interpolation nodes $\bI$ based on our previous discussion. Two settings for updating $\bI$ in the p-AAA algorithm could be considered:
\begin{enumerate}
    \item In each iteration add all points in the set $\bS \cap (\blambda \times \bmu)$ to $\bI$.
    \item In each iteration add only $(\lambda_*,\mu_*)$ to $\bI$.
\end{enumerate}
The first update strategy for $\bI$ is very similar to the original p-AAA algorithm, in the sense that we interpolate data at all combinations of barycentric nodes where we have sample data available. It guarantees that in each iteration of the algorithm a new node will be added to $\blambda$ or $\bmu$, and that $\bI$ increases by one or multiple values. The second update strategy only adds individual interpolation points to $\bI$. In our numerical experiments we did not find any examples where the second update strategy lead to notably better rational approximants than the first one for the same resulting order. Additionally, using a larger interpolation set $\bI$ leads to less degrees of freedom in the LS problem in \eqref{eq:scatteredInterpolationLS}. Therefore, the first update strategy allows for solving the LS problem faster as we typically add several points at once to $\bI$ in this case. For these reasons we specifically use the updating strategy 1. in our proposed p-AAA variation for scattered data sets which is summarized in Algorithm~\ref{alg:nonGridpaaa}.
\begin{algorithm}[t]
    \caption{Scattered data p-AAA}\label{alg:nonGridpaaa}
    \begin{algorithmic}[1]
        \Require{Set of sampling points $\bS$, and samples $\bD$}
        \Ensure{$\br \approx \bff$}
        \State{Initialize $\br \equiv \operatorname{mean}(\bF)$}
        \State{$\blambda,\bmu,\bI \gets \emptyset$}
        \While{error $>$ desired tolerance}
        \State{Determine $ \left( \lambda_*, \mu_* \right) $ via the greedy selection in \eqref{eq:nonGridGreedy}}
        \State Update the sets of barycentric nodes:
        \State\hspace{\algorithmicindent} $\blambda \gets \blambda \cup \{ \lambda_* \}$
        \State\hspace{\algorithmicindent} $\bmu \gets \bmu \cup \{ \mu_* \}$
        \State Update the interpolation points:
        \State\hspace{\algorithmicindent} $\bI \gets \bI \cup \left( \bS \cap \left( \blambda \times \bmu \right) \right)$
        \State{Solve the LS problem in \eqref{eq:scatteredInterpolationLS} for $\alpha, \beta_{\bu}$}
        \State{Compute $\beta_{\bc} = \alpha_{\bc} \circ \bH$ and update $\beta$ via \eqref{eq:betaReconstruction}}
        \State{Update the rational approximant $\br$ via $\alpha, \beta, \blambda$ and $\bmu$}
        \State{error $ \gets \max_{i} \; \left\lvert \bff(X_i,Y_i) - \br\left(X_i,Y_i\right) \right\vert / \max_{i} \;  \left\lvert \bff(X_i,Y_i) \right\rvert $}
        \EndWhile
    \end{algorithmic}
\end{algorithm}

\subsection{Extension to multiple variables}
\label{sec:multiplevariables}
So far, we have considered rational approximation of a function in two variables in order to allow for a more clear and concise presentation of our results. Now, we discuss how the results in Section~\ref{sec:scatteredLSInterpolation} and Section~\ref{sec:scatteredPAAA} can be extended to the case where $\bff$ is a function of $d$ variables $\bz = (z^{(1)},\ldots,z^{(d)}) \in \C^d$. First, we introduce the set of scattered sampling points
\begin{equation*}
    \bS = \left\{ (Z_1^{(1)},\ldots,Z_1^{(d)}), \ldots, (Z_K^{(1)},\ldots,Z_K^{(d)})\right\} \subset \C^d,
\end{equation*}
and the associated set of samples
\begin{equation*}
    \bD = \left\{ \bff\left(z^{(1)},\ldots,z^{(d)}\right) \; | \; (z^{(1)},\ldots,z^{(d)}) \in \bS \right\} \subset \C.
\end{equation*}
Next, consider the set of barycentric nodes
\begin{equation*}
    \blambda^{(1)} \times \cdots \times \blambda^{(d)} \subset \C^d
\end{equation*}
with
\begin{equation*}
    \blambda^{(j)} = \left\{ \lambda_1^{(j)}, \ldots, \lambda_{n_j}^{(j)} \right\} \subset \C \quad \text{for} \quad j = 1,\ldots,d.
\end{equation*}
The barycentric form of a $d$-variate rational function $\br$ is defined along the lines of $\eqref{eq:nd}$ where now we have
\begin{align*}
    \bn(z^{(1)},\ldots,z^{(d)}) &= \sum_{i_1=1}^{n_1} \cdots \sum_{i_d=1}^{n_d} \beta_{i_1\cdots i_d} g_{\blambda^{(1)}}^{(i_1)}(z^{(1)}) \cdots g_{\blambda^{(d)}}^{(i_d)}(z^{(d)}), \\
    \bd(z^{(1)},\ldots,z^{(d)}) &= \sum_{i_1=1}^{n_1} \cdots \sum_{i_d=1}^{n_d} \alpha_{i_1\cdots i_d} g_{\blambda^{(1)}}^{(i_1)}(z^{(1)}) \cdots g_{\blambda^{(d)}}^{(i_d)}(z^{(d)}),
\end{align*}
and $\alpha,\beta \in \C^{n_1 \times \cdots \times n_d}$ are formed by order-$d$ tensors. The set of interpolation points is given by
\begin{equation*}
    \bI = \left\{ (\Lambda_1^{(1)},\ldots,\Lambda_1^{(d)}), \ldots, (\Lambda_k^{(1)},\ldots,\Lambda_k^{(d)})\right\} \subset \blambda^{(1)} \times \cdots \times \blambda^{(d)}.
\end{equation*}
In the following, our goal is to arrive at an optimization problem along the lines of \eqref{eq:scatteredInterpolationLS} for the $d$-variable case. As in the two-variable case, 
\begin{equation*}
    \beta_{\bc} = \cS_{\bc} \vectorize(\beta) \in \C^{k} \quad \text{and} \quad \beta_{\bu} = \cS_{\bu} \vectorize(\beta) \in \C^{n_1 \cdots n_d - k}
\end{equation*}
are formed by the constrained and unconstrained entries of $\beta$, respectively. The underlying matrices $\cS_{\bc} \in \C^{k \times n_1 \cdots n_d}$ and $\cS_{\bu} \in \C^{(n_1\cdots n_d - k) \times n_1 \cdots n_d}$ are defined in Appendix~\ref{sec:selectionmultivar}. This leads to the admissible set of rational functions
\begin{equation*}
    \cR_{SI} = \left\{ \frac{\bn}{\bd} \; \Big| \; \alpha \in \C^{n_1 \times \cdots \times n_d}, \; \beta_{\bu} \in \C^{n_1 \cdots n_d - k}, \; \beta_{\bc} = \alpha_{\bc} \circ \bH \; \text{and} \; \fronorm{\alpha}^2 + \lVert\beta_{\bu}
   \rVert^2 = 1  \right\},
\end{equation*}
where $\bH = \left[ \bff(\Lambda_1^{(1)},\ldots,\Lambda_1^{(d)}), \ldots, \bff(\Lambda_k^{(1)},\ldots,\Lambda_k^{(d)}) \right]^\top$ is the vector of interpolated samples. Finally, define the matrices
\begin{equation*}
    \cD = \diag\left(\bff(Z_1^{(1)},\ldots,Z_1^{(d)}),\ldots,\bff(Z_K^{(1)},\ldots,Z_K^{(d)})\right),
\end{equation*}
and
\begin{equation*}
    \cC = \cC_{\blambda^{(1)}}(\bZ^{(1)}) \odot \cdots \odot \cC_{\blambda^{(d)}}(\bZ^{(d)})
\end{equation*}
with
\begin{equation*}
    \bZ^{(j)} = \left[ Z_1^{(j)}, \ldots, Z_K^{(j)} \right]^\top \quad \text{for} \quad j = 1,\ldots,d. 
\end{equation*}
With these quantities in hand we arrive at the optimization problem
\begin{equation*}
    \min_{\bn / \bd \in \cR_{SI}} \sum_{i=1}^{K} \left\lvert \bff(Z_i^{(1)},\ldots,Z_i^{(d)})\bd(Z_i^{(1)},\ldots,Z_i^{(d)}) - \bn(Z_i^{(1)},\ldots,Z_i^{(d)})\right\rvert^2 = \min_{\fronorm{\alpha}^2 + \lVert\beta_{\bu}
   \rVert^2 = 1 } \left\lVert \M \begin{bmatrix}
        \vectorize(\alpha) \\ \beta_{\bu}
    \end{bmatrix} \right\rVert_2^2,
\end{equation*}
where
\begin{equation*}
    \M = \begin{bmatrix}
        \cD \cC^\top - \cC^\top \cH, & -\cC^\top \cS_{\bu}^\top
    \end{bmatrix} \in \C^{K \times (2 n_1 \cdots n_d - k)}.
\end{equation*}
Again, we can compute the solution to the derived optimization problem via an SVD of $\M$ and reconstruct the full tensor $\beta$ afterwards. As in the two-variable p-AAA formulation, we can select barycentric nodes via a greedy selection
\begin{equation*}
    (\lambda_*^{(1)}, \ldots, \lambda_*^{(d)}) = \argmax_{\bz \in \bS} \left\lvert \bff(\bz) - \br(\bz) \right\rvert
\end{equation*}
and update the interpolation set by adding the points in $\bS \cap (\blambda^{(1)} \times \cdots \times \blambda^{(d)})$ to $\bI$ in each iteration.

\begin{figure}[htp]
    \centering
    \includegraphics{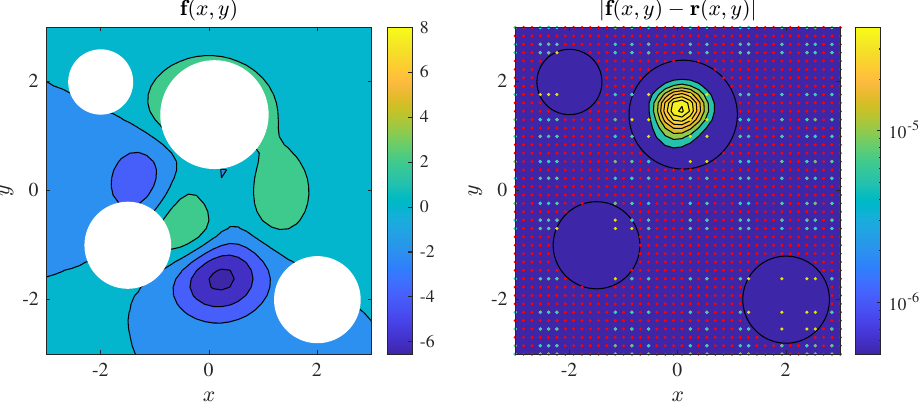}
    \caption{Algorithm~\ref{alg:nonGridpaaa} applied to the peaks function with several gaps as depicted in the left subfigure. On the right, we see the scattered p-AAA approximation error as well as sampling points (red dots), interpolated data (cyan dots) and barycentric nodes which are not interpolated (yellow dots).}
    \label{fig:paaaScatteredPeaks}
\end{figure} %

\section{Numerical examples}
\label{sec:numerics}
The goal of this section is to further illustrate how p-AAA on scattered data sets operates in comparison to the standard p-AAA algorithm. Additionally, we demonstrate that our proposed method is effective on challenging benchmark problems.

\subsection{Peaks function with gaps}
\label{sec:peaksGaps}
Here we reconsider the peaks function which we approximated with p-AAA in Example~\ref{ex:paaaPeaks}. In order to demonstrate the effectiveness of scattered p-AAA, we remove some data lying in circular subdomains from the sampling grid as illustrated in Figure~\ref{fig:paaaScatteredPeaks}. 
The setup discussed here was first introduced in \cite{nakatsukasa_applications_2025}, where several rational approximations computed by the univariate AAA algorithm \cite{nakatsukasa_aaa_2018} were employed to approximate the missing data in the gaps. This approach is fundamentally different from our approach here using scattered p-AAA in that we only need to execute p-AAA once to obtain a single approximant that takes all of the sample data into account at once. Overall, approximately $22\%$ of the data was removed compared to the setup in Example~\ref{ex:paaaPeaks}. After $18$ iterations, scattered p-AAA computes an order-$(14,15)$ rational approximant of $\bff$. As in the original p-AAA algorithm, the barycentric nodes form a grid. However, interpolation conditions are only enforced at points where sample data is available. In particular, the interpolation points do not form a grid. Note that none of the interpolation points are within the circular gaps for this example since data is not available there. As we would expect, the error is largest in the area where the least amount of data is available. Nonetheless, the approximation remains accurate overall, with a maximum error of less than $10^{-4}$ on \emph{the entire grid data} where the data from the removed circular subdomains are added for testing. 
Despite the large gap in the data, scattered p-AAA still produced an accurate rational approximant.

\subsection{Parametric transfer function}
Next, we consider approximation of a parametric transfer function originating from a model for a micro-electromechanical systems (MEMS) gyroscope \cite{morwiki_modgyro}. The transfer function depends on three variables: 1) The Laplace variable $s \in \C$, 2) a structural parameter $d \in [1,2]$, and 3) a parameter $\theta \in [10^{-7},10^{-5}]$ for the angular velocity of the device. The operating frequency range corresponds to $s \in [2\pi \times
0.025, 2\pi \times 0.25] \imath$ for this example. 

For our numerical computations, we sample the three variables of the transfer function on a grid with $200$ linearly spaced samples for the Laplace variable, $20$ linearly spaced samples for the parameter $d$, and $20$ logarithmically sampled values for the parameter $\theta$. Based on these grid data, we compute the original (grid data) p-AAA approximation. Then we remove all data associated with $20\%$, $40\%$, and $60\%$ of randomly chosen parameter combinations. Hence, we assume access to the same $200$ sampling points for the Laplace variable $s$, but the parameter samples for $\theta$ and $d$ form a scattered set. For all of these settings, we compute the scattered data p-AAA approximation and evaluate the approximation error on the full grid of available data. In Figure~\ref{fig:paaaScatteredGyro} we visualize the magnitude of the frequency response at the parameter combinations that resulted in the maximum approximation error. The left-most plot shows the frequency responses at the parameter values that led to the largest error for $20\%$ of the data removed, the middle plot for $40\%$, and the right-most plot for $60\%$. For all setups we ran p-AAA for scattered data with a convergence tolerance of $10^{-2}$ for the relative maximum error. We see that even for the case where $60\%$ of the data was removed, the approximation is still very accurate, as most of the dominant peaks are captured. The final orders of rational approximants were $(18,6,11)$ using full data, $(22,9,12)$ with $20\%$ of data removed, $(16,9,9)$ with $40\%$ of data removed, and $(14,8,8)$ with $60\%$ of data removed. In general, the approximation is worse at higher frequencies. This is primarily due to the fact that the p-AAA greedy selection is based on maximum and not relative maximum errors and the frequency response is rather small for large values of $s$. However, scattered p-AAA still manages to achieve a similar accuracy to p-AAA, even after $60\%$ of the sample data is removed.

\begin{figure}
    \centering
    \includegraphics{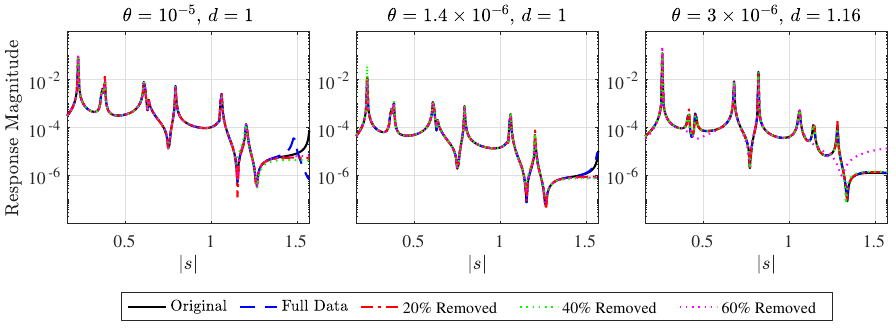}
    \caption{Approximations at parameter combinations with largest approximation error for 20\%, 40\% and 60\% of data removed.}
    \label{fig:paaaScatteredGyro}
\end{figure}

\subsection{Two-parameter thermal model}
In our final numerical example we consider a stationary version of the time-dependent thermal model introduced in \cite{rave_non-stationary_2021} (see \cite{balicki_multivariate_2025} for a description of the stationary problem). For this model the output is the average temperature in a square domain at the steady-state of the model, which depends on four parameters $(p^{(1)},p^{(2)},p^{(3)},p^{(4)}) \in \R^4$ that specify the heat conductivity on four disjoint subdomains. Here we use the fixed values $p^{(4)}=10^{-6}$ and $p^{(3)}=10^{2}$ and use $p^{(1)},p^{(2)} \in \left[ 10^{-6},10^2 \right]$. We randomly choose $K=50$ scattered sampling points in $\left[ 10^{-6},10^2 \right]^2$ to generate our training data and use Algorithm~\ref{alg:nonGridpaaa} to approximate the model output $\bff(x,y)$ where $x=p^{(1)}$ and $y=p^{(2)}$ via a bivariate rational function. Only three iterations were necessary to compute a highly accurate rational approximation of order $(2,2)$. The approximated function and sampling points are shown in the left plot of Figure~\ref{fig:paaaScatteredThermal}. The right plot displays the approximation error and the three selected interpolation points. The approximant was evaluated using the test data generated on a full $40 \times 40$ grid uniformly covering $[10^{-4},10^8]^2$. This shows that the approximation quality is high even at parameter values that differ from the training data by two orders of magnitude.

\begin{figure}
    \centering
    \includegraphics{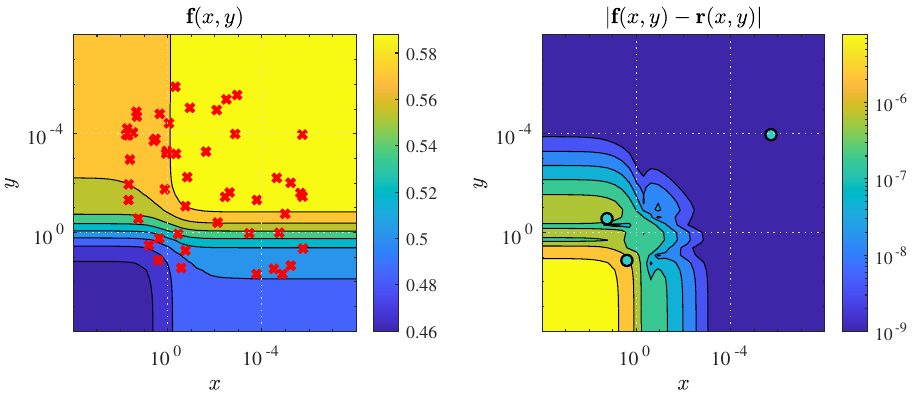}
    \caption{Approximation of a stationary thermal model with two parameters using $K=50$ sampling points (red crosses). Only three interpolation points (cyan dots) were selected by scattered p-AAA leading to an accurate rational approximation.}
    \label{fig:paaaScatteredThermal}
\end{figure}

\section{Conclusions}
\label{sec:conclusions}
We introduced several formulations of linear LS problems that can be used in multivariate rational approximation algorithms and describe how closed-form solutions via an SVD can be computed for them. Most importantly, this includes an optimization problem that incorporates interpolation conditions for scattered sets of interpolation points. This leads to a formulation of the p-AAA algorithm that operates on scattered data. We show that our proposed approach is effective via several numerical examples. While our work focuses on rational approximation, future research could investigate the exact recovery of a rational function from scattered data sets. The parametric Loewner framework \cite{ionita2014} achieves this using grid data, however, to our knowledge, extensions to scattered data have not been investigated yet. It is well-known from the theory of polynomial interpolation, that special types of scattered data sets (e.g., Padua points \cite{bos_bivariate_2006}) are required to guarantee the uniqueness of bivariate polynomial interpolants. Extending these considerations to bivariate rational functions would be an interesting next step.

\section*{Acknowledgments}
This work was supported in part by the National Science Foundation (NSF), United States under Grant No. DMS-2411141. We thank Nick Trefethen for sharing the code used to generate the data for the example in Section~\ref{sec:peaksGaps}, and Nick Trefethen, Yuji Nakatsukasa, and Heather Wilber for many insightful discussions regarding this example.

\bibliographystyle{plainurl}
\bibliography{references}

\begin{thebibliography}{10}

\bibitem{ackermann_second-order_2025}
M.~S. Ackermann, I.~V. Gosea, S.~Gugercin, and S.~W.~R. Werner.
\newblock Second-order {AAA} algorithms for structured data-driven modeling.
\newblock arXiv preprint, 2025.
\newblock \href {https://doi.org/10.48550/arXiv.2506.02241} {\path{doi:10.48550/arXiv.2506.02241}}.

\bibitem{antoulas_interpolatory_2020}
A.~C. Antoulas, C.~A. Beattie, and S.~Güğercin.
\newblock {\em Interpolatory methods for model reduction}.
\newblock SIAM, Society for Industrial and Applied Mathematics, Philadelphia, 2020.
\newblock \href {https://doi.org/10.1137/1.9781611976083} {\path{doi:10.1137/1.9781611976083}}.

\bibitem{antoulas_loewner_2024}
A.~C. Antoulas, I.~V. Gosea, and C.~Poussot-Vassal.
\newblock The {Loewner} framework for parametric systems: Taming the curse of dimensionality.
\newblock arXiv preprint, 2024.
\newblock \href {https://doi.org/10.48550/arXiv.2405.00495} {\path{doi:10.48550/arXiv.2405.00495}}.

\bibitem{antoulas2025}
A.~C. Antoulas, I.~V. Gosea, and C.~Poussot-Vassal.
\newblock On the {Loewner} framework, the {Kolmogorov} superposition theorem, and the curse of dimensionality.
\newblock arXiv preprint, 2025.
\newblock \href {https://doi.org/10.48550/arXiv.2405.00495} {\path{doi:10.48550/arXiv.2405.00495}}.

\bibitem{antoulas_chapter_2017}
A.~C. Antoulas, S.~Lefteriu, and A.~C. Ionita.
\newblock Chapter 8: A tutorial introduction to the {Loewner} framework for model reduction.
\newblock In {\em Model {Reduction} and {Approximation}}, pages 335--376. SIAM, 2017.
\newblock \href {https://doi.org/10.1137/1.9781611974829.ch8} {\path{doi:10.1137/1.9781611974829.ch8}}.

\bibitem{balicki_multivariate_2025}
L.~Balicki and S.~Gugercin.
\newblock Multivariate rational approximation via low-rank tensors and the {p-AAA} algorithm.
\newblock arXiv preprint, 2025.
\newblock \href {https://doi.org/10.48550/arXiv.2502.03204} {\path{doi:10.48550/arXiv.2502.03204}}.

\bibitem{beattie_realization-independent_2012}
C.~Beattie and S.~Gugercin.
\newblock Realization-independent $\mathcal{H}_2$-approximation.
\newblock In {\em 2012 {IEEE} 51st {IEEE} {Conference} on {Decision} and {Control} ({CDC})}, pages 4953--4958, 2012.
\newblock \href {https://doi.org/10.1109/CDC.2012.6426344} {\path{doi:10.1109/CDC.2012.6426344}}.

\bibitem{benner_survey_2015}
P.~Benner, S.~Gugercin, and K.~Willcox.
\newblock A survey of projection-based model reduction methods for parametric dynamical systems.
\newblock {\em SIAM Review}, 57:483--531, 2015.
\newblock \href {https://doi.org/10.1137/130932715} {\path{doi:10.1137/130932715}}.

\bibitem{berljafa_rkfit_2017}
M.~Berljafa and S.~Güttel.
\newblock The {RKFIT} algorithm for nonlinear rational approximation.
\newblock {\em SIAM Journal on Scientific Computing}, 39:A2049--A2071, 2017.
\newblock \href {https://doi.org/10.1137/15M1025426} {\path{doi:10.1137/15M1025426}}.

\bibitem{borghi_mathcal_2024}
A.~Borghi and T.~Breiten.
\newblock $\mathcal{H}_2$ optimal rational approximation on general domains.
\newblock {\em Advances in Computational Mathematics}, 50:28, 2024.
\newblock \href {https://doi.org/10.1007/s10444-024-10125-8} {\path{doi:10.1007/s10444-024-10125-8}}.

\bibitem{bos_bivariate_2006}
L.~Bos, M.~Caliari, S.~De~Marchi, M.~Vianello, and Y.~Xu.
\newblock Bivariate {Lagrange} interpolation at the {Padua} points: The generating curve approach.
\newblock {\em Journal of Approximation Theory}, 143:15--25, 2006.
\newblock \href {https://doi.org/10.1016/j.jat.2006.03.008} {\path{doi:10.1016/j.jat.2006.03.008}}.

\bibitem{boulle2024}
N.~Boullé, A.~Herremans, and D.~Huybrechs.
\newblock Multivariate rational approximation of functions with curves of singularities.
\newblock {\em SIAM Journal on Scientific Computing}, 46:A3401--A3426, 2024.
\newblock \href {https://doi.org/10.1137/23M1626414} {\path{doi:10.1137/23M1626414}}.

\bibitem{brennan_contour_2023}
M.~C. Brennan, M.~Embree, and S.~Gugercin.
\newblock Contour integral methods for nonlinear eigenvalue problems: A systems theoretic approach.
\newblock {\em SIAM Review}, 65:439--470, 2023.
\newblock \href {https://doi.org/10.1137/20M1389303} {\path{doi:10.1137/20M1389303}}.

\bibitem{hujdurovic_aaa-least_2023}
S.~Costa and L.~N. Trefethen.
\newblock {AAA}-least squares rational approximation and solution of {Laplace} problems.
\newblock In {\em European {Congress} of {Mathematics}}, pages 511--534. EMS Press, 1 edition, 2023.
\newblock \href {https://doi.org/10.4171/8ecm/16} {\path{doi:10.4171/8ecm/16}}.

\bibitem{deschrijver_robust_2008}
D.~Deschrijver, T.~Dhaene, and D.~De~Zutter.
\newblock Robust parametric macromodeling using multivariate orthonormal vector fitting.
\newblock {\em IEEE Transactions on Microwave Theory and Techniques}, 56:1661--1667, 2008.
\newblock \href {https://doi.org/10.1109/TMTT.2008.924346} {\path{doi:10.1109/TMTT.2008.924346}}.

\bibitem{drmac_quadrature-based_2015}
Z.~Drmač, S.~Gugercin, and C.~Beattie.
\newblock Quadrature-based vector fitting for discretized $\mathcal{H}_2$ approximation.
\newblock {\em SIAM Journal on Scientific Computing}, 37:A625--A652, 2015.
\newblock \href {https://doi.org/10.1137/140961511} {\path{doi:10.1137/140961511}}.

\bibitem{feng_posteriori_2017}
L.~Feng, A.~C. Antoulas, and P.~Benner.
\newblock Some a posteriori error bounds for reduced-order modelling of (non-)parametrized linear systems.
\newblock {\em ESAIM: Mathematical Modelling and Numerical Analysis}, 51:2127--2158, 2017.
\newblock \href {https://doi.org/10.1051/m2an/2017014} {\path{doi:10.1051/m2an/2017014}}.

\bibitem{gustavsen_rational_1999}
B.~Gustavsen and A.~Semlyen.
\newblock Rational approximation of frequency domain responses by vector fitting.
\newblock {\em IEEE Transactions on Power Delivery}, 14:1052--1061, 1999.
\newblock \href {https://doi.org/10.1109/61.772353} {\path{doi:10.1109/61.772353}}.

\bibitem{guttel_robust_2022}
S.~Güttel, G.~M. Negri~Porzio, and F.~Tisseur.
\newblock Robust rational approximations of nonlinear eigenvalue problems.
\newblock {\em SIAM Journal on Scientific Computing}, 44:A2439--A2463, 2022.
\newblock \href {https://doi.org/10.1137/20M1380533} {\path{doi:10.1137/20M1380533}}.

\bibitem{hokanson2017projected}
J.~M. Hokanson.
\newblock Projected nonlinear least squares for exponential fitting.
\newblock {\em SIAM Journal on Scientific Computing}, 39(6):A3107--A3128, 2017.
\newblock \href {https://doi.org/10.1137/16M1084067} {\path{doi:10.1137/16M1084067}}.

\bibitem{hokanson2020}
J.~M. Hokanson.
\newblock Multivariate rational approximation using a stabilized {Sanathanan}-{Koerner} iteration.
\newblock arXiv preprint, 2020.
\newblock \href {https://doi.org/10.48550/arXiv.2009.10803} {\path{doi:10.48550/arXiv.2009.10803}}.

\bibitem{hund_optimization-based_2022}
M.~Hund, T.~Mitchell, P.~Mlinarić, and J.~Saak.
\newblock Optimization-based parametric model order reduction via $\mathcal{H}_2 \otimes \mathcal{L}_2$ first-order necessary conditions.
\newblock {\em SIAM Journal on Scientific Computing}, 44:A1554--A1578, 2022.
\newblock \href {https://doi.org/10.1137/21M140290X} {\path{doi:10.1137/21M140290X}}.

\bibitem{ionita2014}
A.~C. Ionita and A.~C. Antoulas.
\newblock Data-driven parametrized model reduction in the {Loewner} framework.
\newblock {\em SIAM Journal on Scientific Computing}, 36:A984--A1007, 2014.
\newblock \href {https://doi.org/10.1137/130914619} {\path{doi:10.1137/130914619}}.

\bibitem{kolda_tensor_2009}
T.~G. Kolda and B.~W. Bader.
\newblock Tensor decompositions and applications.
\newblock {\em SIAM Review}, 51:455--500, 2009.
\newblock \href {https://doi.org/10.1137/07070111X} {\path{doi:10.1137/07070111X}}.

\bibitem{levy_complex-curve_1959}
E.~C. Levy.
\newblock Complex-curve fitting.
\newblock {\em IRE Transactions on Automatic Control}, AC-4:37--43, 1959.
\newblock \href {https://doi.org/10.1109/TAC.1959.6429401} {\path{doi:10.1109/TAC.1959.6429401}}.

\bibitem{lietaert_automatic_2022}
P.~Lietaert, K.~Meerbergen, J.~Pérez, and B.~Vandereycken.
\newblock Automatic rational approximation and linearization of nonlinear eigenvalue problems.
\newblock {\em IMA Journal of Numerical Analysis}, 42:1087--1115, 2022.
\newblock \href {https://doi.org/10.1093/imanum/draa098} {\path{doi:10.1093/imanum/draa098}}.

\bibitem{mlinaric_interpolatory_2024}
P.~Mlinarić, P.~Benner, and S.~Gugercin.
\newblock Interpolatory necessary optimality conditions for reduced-order modeling of parametric linear time-invariant systems.
\newblock arXiv preprint, 2024.
\newblock \href {https://doi.org/10.48550/arXiv.2401.10047} {\path{doi:10.48550/arXiv.2401.10047}}.

\bibitem{nakatsukasa_aaa_2018}
Y.~Nakatsukasa, O.~Sète, and L.~N. Trefethen.
\newblock The {AAA} algorithm for rational approximation.
\newblock {\em SIAM Journal on Scientific Computing}, 40:A1494--A1522, 2018.
\newblock \href {https://doi.org/10.1137/16M1106122} {\path{doi:10.1137/16M1106122}}.

\bibitem{nakatsukasa_applications_2025}
Y.~Nakatsukasa and L.~N. Trefethen.
\newblock Applications of {AAA} rational approximation.
\newblock arXiv preprint, 2025.
\newblock \href {https://doi.org/10.48550/arXiv.2510.16237} {\path{doi:10.48550/arXiv.2510.16237}}.

\bibitem{rave_non-stationary_2021}
S.~Rave and J.~Saak.
\newblock A non-stationary thermal-block benchmark model for parametric model order reduction.
\newblock In {\em Model {Reduction} of {Complex} {Dynamical} {Systems}}, pages 349--356. Springer International Publishing, Cham, 2021.
\newblock \href {https://doi.org/10.1007/978-3-030-72983-7_16} {\path{doi:10.1007/978-3-030-72983-7_16}}.

\bibitem{rodriguez2023}
A.~Carracedo Rodriguez, L.~Balicki, and S.~Gugercin.
\newblock The {p-AAA} algorithm for data-driven modeling of parametric dynamical systems.
\newblock {\em SIAM Journal on Scientific Computing}, 45:A1332--A1358, 2023.
\newblock \href {https://doi.org/10.1137/20M1322698} {\path{doi:10.1137/20M1322698}}.

\bibitem{rubin_bounding_2022}
D.~Rubin, A.~Townsend, and H.~Wilber.
\newblock Bounding {Zolotarev} numbers using {Faber} rational functions.
\newblock {\em Constructive Approximation}, 56:207--232, 2022.
\newblock \href {https://doi.org/10.1007/s00365-022-09585-2} {\path{doi:10.1007/s00365-022-09585-2}}.

\bibitem{sanathanan_transfer_1963}
C.~Sanathanan and J.~Koerner.
\newblock Transfer function synthesis as a ratio of two complex polynomials.
\newblock {\em IEEE Transactions on Automatic Control}, 8:56--58, 1963.
\newblock \href {https://doi.org/10.1109/TAC.1963.1105517} {\path{doi:10.1109/TAC.1963.1105517}}.

\bibitem{morwiki_modgyro}
{The MORwiki Community}.
\newblock Modified gyroscope.
\newblock {MORwiki} -- Model Order Reduction Wiki, 2018.
\newblock URL: \url{https://modelreduction.org/morwiki/Modified_Gyroscope}.

\bibitem{whitfield_asymptotic_1987}
A.~H. Whitfield.
\newblock Asymptotic behaviour of transfer function synthesis methods.
\newblock {\em International Journal of Control}, 45:1083--1092, 1987.
\newblock \href {https://doi.org/10.1080/00207178708933791} {\path{doi:10.1080/00207178708933791}}.

\bibitem{wilber_data-driven_2022}
H.~Wilber, A.~Damle, and A.~Townsend.
\newblock Data-driven algorithms for signal processing with trigonometric rational functions.
\newblock {\em SIAM Journal on Scientific Computing}, 44:C185--C209, 2022.
\newblock \href {https://doi.org/10.1137/21M1420277} {\path{doi:10.1137/21M1420277}}.

\end{thebibliography}

\appendix

\section{Selection operators}
Here we give precise definitions of the selection operators $\cS_{\bc}$ and $\cS_{\bu}$ used for scattered interpolation. We derive these operators for the case of two variables and then generalize the description to $d$ variables.
\subsection{Selection operators for two variables}
\label{sec:selectiontwovar}
First, we define the masking matrices $\bM_{\bc}, \bM_{\bu} \in \C^{n \times m}$ via
\begin{equation*}
    \left(\bM_{\bc}\right)_{ij} = \begin{cases}
        1 \quad & \text{if } (\lambda_i,\mu_j) \in \bI \\
        0 \quad & \text{else}
    \end{cases} \quad \text{and} \quad \left(\bM_{\bu}\right)_{ij} = \begin{cases}
        0 \quad & \text{if } (\lambda_i,\mu_j) \in \bI \\
        1 \quad & \text{else}
    \end{cases}
\end{equation*}
Next, we define the index sets
\begin{equation*}
    \begin{aligned}
    \{ \bc_1,\ldots,\bc_k \} &= \{ i \in \N \, | \, \vectorize(\bM_{\bc})_i = 1 \}, \\ 
    \{ \bu_1,\ldots,\bu_{nm - k} \} &= \{ i \in \N \, | \, \vectorize(\bM_{\bu})_i = 1 \},
    \end{aligned}
\end{equation*}
and the matrices $\cS_{\bc} \in \C^{k \times nm}$ and $\cS_{\bu} \in \C^{(nm-k) \times nm}$ via
\begin{equation*}
    \left(\cS_{\bc}\right)_{ij} = \begin{cases}
        1 \quad &\text{if } \bc_i = j \\
        0 \quad &\text{else}
    \end{cases} \qquad \text{and} \qquad \left(\cS_{\bu}\right)_{ij} = \begin{cases}
        1 \quad &\text{if } \bu_i = j \\
        0 \quad &\text{else}
    \end{cases}
\end{equation*}
\subsection{Selection operators for $d$ variables}
\label{sec:selectionmultivar}
To extend the previously introduced selection operators to $d$ variables we consider $\bM_{\bc}, \bM_{\bu} \in \C^{n_1 \times \cdots \times n_d}$ with
\begin{equation*}
    \left(\bM_{\bc}\right)_{i_1 \dots i_d} = \begin{cases}
        1 \quad & \text{if } (\lambda_{i_1},^{(1)}\ldots,\lambda_{i_d}^{(d)}) \in \bI \\
        0 \quad & \text{else}
    \end{cases} \quad \text{and} \quad \left(\bM_{\bu}\right)_{i_1 \dots i_d} = \begin{cases}
        0 \quad & \text{if } (\lambda_{i_1},^{(1)}\ldots,\lambda_{i_d}^{(d)}) \in \bI \\
        1 \quad & \text{else}
    \end{cases}
\end{equation*}
Next, we define
\begin{equation*}
    \begin{aligned}
    \{ \bc_1,\ldots,\bc_k \} &= \{ i \in \N \, | \, \vectorize(\bM_{\bc})_i = 1 \}, \\ 
    \{ \bu_1,\ldots,\bu_{nm - k} \} &= \{ i \in \N \, | \, \vectorize(\bM_{\bu})_i = 1 \},
    \end{aligned}
\end{equation*}
and the matrices $\cS_{\bc} \in \C^{k \times n_1 \cdots n_d}$ and $\cS_{\bu} \in \C^{(n_1 \cdots n_d - k) \times n_1 \cdots n_d}$ via
\begin{equation*}
    \left(\cS_{\bc}\right)_{ij} = \begin{cases}
        1 \quad &\text{if } \bc_i = j \\
        0 \quad &\text{else}
    \end{cases} \qquad \text{and} \qquad \left(\cS_{\bu}\right)_{ij} = \begin{cases}
        1 \quad &\text{if } \bu_i = j \\
        0 \quad &\text{else}
    \end{cases}
\end{equation*}

\end{document}